\documentclass[oneside,english]{amsart}
\usepackage[T1]{fontenc}
\usepackage[latin9]{inputenc}
\usepackage{geometry}
\usepackage{amstext}
\usepackage{amsthm}
\usepackage{amssymb}
\usepackage{color,fancybox}
\usepackage{amsrefs}

\makeatletter
\numberwithin{equation}{section}
\numberwithin{figure}{section}
\theoremstyle{plain}
\newtheorem{thm}{\protect\theoremname}
  \theoremstyle{plain}
  \newtheorem{fact}[thm]{\protect\factname}
  \theoremstyle{remark}
  \newtheorem{rem}[thm]{\protect\remarkname}
  \theoremstyle{definition}
  \newtheorem{defn}[thm]{\protect\definitionname}
  \theoremstyle{plain}
  \newtheorem{cor}[thm]{\protect\corollaryname}
  \theoremstyle{plain}
	\newtheorem{lem}[thm]{\protect\lemmaname}
  \theoremstyle{plain}
	\newtheorem{prop}[thm]{\protect\propositionname}
  \theoremstyle{definition}
  \newtheorem{example}[thm]{\protect\examplename}

\newtheorem{obs}{Assumption}
\renewenvironment{fact}{\begin{obs}}{\end{obs}}
\theoremstyle{fact}

\makeatother

\usepackage{babel}
  \providecommand{\corollaryname}{Corollary}
  \providecommand{\definitionname}{Definition}
  \providecommand{\examplename}{Example}
	\providecommand{\lemmaname}{Lemma}
  \providecommand{\factname}{Fact}
  \providecommand{\propositionname}{Proposition}
  \providecommand{\remarkname}{Remark}
\providecommand{\theoremname}{Theorem}

\begin{document}

\title[Dynamical boundary conditions in non-cylindrical domains]{Dynamical boundary conditions in a non-cylindrical domain for the Laplace equation}

\author{Pedro T. P. Lopes}
\address{Instituto de Matem\'atica e Estat\'istica, Universidade de S\~ao Paulo,  Rua do Mat\~ao 1010, 05508-090, S\~ao Paulo, SP, Brazil}
\email{pplopes@ime.usp.br}
\thanks{Pedro T. P. Lopes was partially supported by FAPESP 2016/07016-8}

\author{Marcone C. Pereira}
\address{Instituto de Matem\'atica e Estat\'istica, Universidade de S\~ao Paulo,  Rua do Mat\~ao 1010, 05508-090, S\~ao Paulo, SP, Brazil}
\email{marcone@ime.usp.br}
\thanks{Marcone C. Pereira was partially supported by CNPq-Brazil 471210/2013-7 and FAPESP 2017/02630-2}



\date{\today}

\begin{abstract}

In this paper, we study existence, uniqueness and asymptotic behavior of the Laplace equation with dynamical boundary conditions on regular 
non-cylindrical domains. We write the problem as a non-autonomous Dirichlet-to-Neumann operator and use form methods in a more general framework to accomplish our goal. A class of non-autonomous elliptic problems with dynamical boundary conditions on Lipschitz domains is also considered in this same context.

\end{abstract}


\maketitle
\tableofcontents{}

\section{Introduction}

In this paper, we consider the following problem:
\begin{equation} \label{eq:MainProblem}
\left\{
\begin{gathered}
\left(\lambda+\Delta\right)u\left(t,x\right) = 0,\,(t,x)\in D, t> t_{0},\\
\frac{\partial u}{\partial t}\left(t,x\right) = -\frac{\partial u}{\partial n}\left(t,x\right)+f\left(t,x\right),\,(t,x)\in S, t> t_{0}, \\
u\left(t_{0},x\right) = u_{0}\left(x\right),\,x\in\partial\Omega_{t_{0}},
\end{gathered}
\right.
\end{equation}
where $\lambda<0$ is a constant, $t_{0}\ge0$, $D\subset\mathbb{R}^{n+1}$ is an appropriate open set of real variables $(t,x)=(t,x_{1},...,x_{n})$ bounded by
a bounded domain $\Omega_{0}\subset\mathbb{R}^{n}$ at $t=0$ and
a $n$-dimensional surface $S$ on the half space $t>0$. We denote
by $\Omega_{\tau}$ and by $\partial\Omega_{\tau}$, $\tau>0$, the intersections
$D\cap\{\left(t,x\right)\in\mathbb{R}^{n+1};\,t=\tau\}$ and $S\cap\{\left(t,x\right)\in\mathbb{R}^{n+1};\,t=\tau\}$,
respectively. Therefore $D=\cup_{t>0}\Omega_{t}$ and $S=\cup_{t>0}\partial\Omega_{t}$.
The outward normal derivative at the point $x\in\partial\Omega_{t}$ is denoted by $\frac{\partial u}{\partial n}\left(t,x\right)$, and $n$ is the unit outward normal vector to the boundary $\partial \Omega_t$.

Our main goal is first to show the existence and uniqueness of the solutions of this Laplace equation with dynamical boundary conditions.
Notice that this task is not trivial since problem \eqref{eq:MainProblem} is posed in a non-cylindrical domain $D$. 
Furthermore, we study the asymptotic behavior of the solutions at infinite time, when $\Omega_{t}$ converges to a domain $\Omega$, and $f\left(t,.\right)$ converges to a function $f_{\infty}\left(.\right)$.
In this situation, we obtain that the solutions converge to the stationary problem 
\begin{equation}
\left\{
\begin{gathered}
\left(\lambda+\Delta\right)u_{\infty}\left(x\right)  =  0, \quad x\in\Omega\\
\frac{\partial u_{\infty}}{\partial n}\left(x\right)  =  f_{\infty}\left(x\right),\quad x\in\partial\Omega
\end{gathered}
\right. .
\nonumber
\end{equation}

We set the non-cylindrical domain $D$ by smooth perturbations of a fixed open set $\Omega$ which are defined by diffeomorphisms according, for instance, to D. Henry in \cite{Henryboundary}.
Performing a change of variable, we transform \eqref{eq:MainProblem} in a non-autonomous Dirichlet-to-Neumann operator posed in the cylindrical set $]t_0,\infty[ \times \Omega$, which leads us to consider non-autonomous elliptic equations with dynamical boundary conditions. We introduce these non-autonomous equations in the following way:

Let $P\left(t,x,D\right)$ be a second order elliptic
operator acting on the bounded domain $\Omega\subset\mathbb{R}^{n}$ given by
\begin{equation}\label{eq:P(t,x,D)}
\begin{array}{l}
\displaystyle P\left(t,x,D\right) := -\sum_{i,j=1}^{n}  \partial_{x_{i}} \left( a_{ij}\left(t,x\right)\partial_{x_{j}}u\left(x\right) \right) +\sum_{j=1}^{n}b_{j}\left(t,x\right)\partial_{x_{j}}u\left(x\right) \\[10pt] 
\qquad \qquad \qquad \qquad \displaystyle -\sum_{j=1}^{n}\partial_{x_{j}}\left(c_{j}\left(t,x\right)u\left(x\right)\right)+d\left(t,x\right)u\left(x\right).
\end{array}
\end{equation}
Suppose that for every suitable function $g$ defined on the boundary
$\partial\Omega$, there exists a unique solution to the Dirichlet problem:
$P\left(t,x,D\right)u=0$ and $\left.u\right|_{\partial\Omega}=g$.
In this case, the Dirichlet-to-Neumann operator, denoted here by $A\left(t\right)$,
is the operator that maps $g$ to the conormal derivative of
$u$ on $\partial\Omega$:
\begin{equation}
A\left(t\right)\left(g\right)=C\left(t,x,D\right)\left(u\right):=\sum_{i,j=1}^{n}a_{ij}\left(t,x\right)\nu_{i}\left(x\right)\partial_{x_{j}}u\left(x\right)+\sum_{j=1}^{n}c_{j}\left(t,x\right)u\left(x\right)\nu_{j}\left(x\right), \label{eq:dirichlettoneumann}
\end{equation}
 where $\nu(x) = (\nu_1(x),...,\nu_n(x))$ denotes the outer unit normal at $x\in \partial \Omega$.

Thereby, we introduce the 
non-autonomous problem defined on the boundary $\partial\Omega$:
\begin{equation} \label{eq:mainproblem}
\begin{array}{rcl}
\frac{du}{dt}\left(t\right)+A\left(t\right)u\left(t\right)&=&f\left(t\right), \; t\ge t_{0} \\
u\left(t_{0}\right)&=&u_{0}
\end{array}. 
\end{equation}

Here, we give simple conditions that guarantee that the above problem is well-posed. In order to do that on $H^{-\frac{1}{2}}(\partial\Omega)$, we use form methods and assume Hölder continuity with exponent $\alpha\in]0,1]$ in time of the forms that define the operators $\left\{ A\left(t\right)\right\} _{t\ge t_{0}}$. We are then able to show that the operators satisfy the Tanabe-Sobolevskii conditions. On $L^{2}(\partial\Omega)$, we assume Hölder continuity of the forms with an exponent $\alpha\in]\frac{1}{2},1]$ and prove that the operators satisfy the Yagi conditions. Our assumptions also allow us to study the asymptotic behavior at infinity of the Problem \eqref{eq:mainproblem}.
Consequently, the results concerned with \eqref{eq:MainProblem} are directly obtained as an application of the results obtained for \eqref{eq:mainproblem}.

Our approach follows W. Arendt and A. Elst \cite{Arendtelst,Arendtformmethods}, see also \cite{Arendtelstrough,Arendtetal}, where form methods were used to deal
with  the Dirichlet-to-Neumann problem defined by the Laplace operator.
Their technique has been also used for general second
order problems by J. Abreu and E. Capelato \cite{AbrelJamilErica}, and E. M. Ouhabaz \cite{ElstOuhabaz}. The novelty here
is to study non-autonomous problems in bounded domains combining form methods for the Dirichlet-to-Neumann problem with the Tanabe-Sobolevskii and the Yagi 
conditions \cite{Tanabe,sobolevskii,Yagibook}.

Notice that dynamical boundary conditions have been studied by many authors. Among them, we mention: J. Lions \cite{Lionsnonlinear}, J. Escher \cite{EscherDNContinuous,Eschernonlineardynamic}, J. Escher
and J. Seiler \cite{EscherSeiler}, T. Hintermann \cite{Hintermann},
A. Friedman and Shinbrot \cite{FriedmanShinbrot} and L. Vazquez and
E. Vitillaro \cite{Vazquezenzo}. Although there exists a big literature on the study of parabolic equations on non-cylindrical domains, among which we can mention the pioneering work of A. Friedman \cite{Friedmannoncylindrical}, as well S. Bonaccorsi and G. Guatteri \cite{BG}, the recent papers by Ma To Fu et al. \cite{MaToFu} and J. Calvo et al. \cite{Calvo}, we could not find any result for the Laplace equation with dynamical boundary conditions on non-cylindrical domains. Thus, we are fulfilling this gap with this work, besides our study of non-autonomous equations with dynamic boundary conditions. It is interesting to note that convergence of the Dirichlet-to-Neumann operators associated with the Laplacian on varying domains was considered by A.F.M. ter Elst and E.M. Ouhabaz \cite{ElstOuhabaz}, but they did not studied it as a change in time and their results are quite different from ours.

The organization of the paper is as follows: In the Section 2, we give a rigorous definition of the Dirichlet-to-Neumann operators on Lipschitz domains and we characterize them as bounded operators from $H^\frac{1}{2}(\partial\Omega)$ to $H^{-\frac{1}{2}}(\partial\Omega)$ using forms. Next, we recall the Tanabe-Sobolevskii conditions discussing how to recover them using form methods, to finally study existence,
uniqueness and asymptotic behavior to Problem (\ref{eq:mainproblem}) on $H^{-\frac{1}{2}}\left(\partial\Omega\right)$. Using the Yagi conditions, we are then able to extend the study to $L^{2}\left(\partial\Omega\right)$ functions. 
In Section 3, we apply the results of Section 2 to a non-autonomous elliptic equation with dynamic boundary conditions in bounded Lipschitz domains, and then, to the Laplace equation with dynamic boundary conditions on non-cylindrical domains.

\section{Non-autonomous Dirichlet-to-Neumann problem}

In this section, $\Omega\subset\mathbb{R}^{n}$ is a Lipschitz bounded
domain. We set some notations and recall
some facts, whose proofs can be found, for instance, in P. Grisvard
\cite{grisvard}. For $m\in\mathbb{Z}$, the Sobolev spaces are defined by:
\[
H^{m}\left(\Omega\right):=\left\{ \begin{array}{ccl}
u\in L^{2}\left(\Omega\right);& &\,\sum_{\left|\alpha\right|\le m}\left\Vert \partial_{x}^{\alpha}u\right\Vert _{L^{2}\left(\Omega\right)}^{2}<\infty,\,{\rm if }\, m\ge0\\
u\in\mathcal{D}'\left(\Omega\right);& &\,\exists u_{\alpha}\in L^{2}\left(\Omega\right),\,\left|\alpha\right|\le |m|,\, {\rm such}\, {\rm that} \, u=\sum_{\left|\alpha\right|\le |m|}\partial_{x}^{\alpha}u_{\alpha}, {\rm if }\, m<0 \quad 
\end{array}\right. .
\]
 The norms are denoted by $u\in H^{m}\left(\Omega\right)\mapsto\left\Vert u\right\Vert _{H^{m}\left(\Omega\right)}$.
The Sobolev space $H^{s}\left(\partial\Omega\right)$, $0<s<1$, is
defined as the space of all measurable functions $u:\partial\Omega\to\mathbb{C}$
such that
\[
\left\Vert u\right\Vert _{H^{s}\left(\partial\Omega\right)}^{2}:=\int_{\partial\Omega}\left|u\left(x\right)\right|^{2}d\sigma_{x}+\int_{\partial\Omega\times\partial\Omega}\frac{\left|u\left(x\right)-u\left(y\right)\right|^{2}}{\left\Vert x-y\right\Vert ^{n-1+2s}}d\sigma_{x}d\sigma_{y}<\infty,
\]
where $\sigma_{x}$ is the surface measure of $\partial\Omega$. We
denote by $C^{m}\left(\overline{\Omega}\right)$ the set of all functions
$u$ such that the derivatives $\partial^{\alpha}u$ exist for all
$\left|\alpha\right|\le m$ and are continuous functions up to the
boundary. The set of functions of class $C^{m}$ whose support is contained in $\Omega$ will be denoted by $C^{m}_{c}\left(\Omega\right)$. The trace operator $\gamma_{0}:H^{1}\left(\Omega\right)\to H^{\frac{1}{2}}\left(\partial\Omega\right)$
is the unique continuous extension of the function $C^{1}\left(\overline{\Omega}\right)\ni u\mapsto \left.u\right|_{\partial\Omega} \in C\left(\partial\Omega\right)$.
As usual we denote $H_{0}^{1}\left(\Omega\right):=\ker\left(\gamma_{0}\right)$. Note that there exists a continuous extension operator $\mathcal{E}:H^{\frac{1}{2}}\left(\partial\Omega\right)\to H^{1}\left(\Omega\right)$
such that $\gamma_{0}\circ \mathcal{E}$ is the identity. Denoting by $\left(.,.\right)_{L^{2}\left(\Omega\right)}$
and by $\left(.,.\right)_{L^{2}\left(\partial\Omega\right)}$ the usual
scalar products of $L^{2}\left(\Omega\right)$ and $L^{2}\left(\partial\Omega\right)$,
respectively, it is well known that $\left(.,.\right)_{L^{2}\left(\Omega\right)}: C^{1}\left(\overline{\Omega}\right)\times C^{1}_{c}\left(\Omega\right) \to\mathbb{C}$
extends uniquely to a sesquilinear form $\left\langle .,.\right\rangle _{H^{-1}\left(\Omega\right)\times H_{0}^{1}\left(\Omega\right)}:H^{-1}\left(\Omega\right)\times H_{0}^{1}\left(\Omega\right)\to\mathbb{C}$
that allows the identification of $H^{-1}\left(\Omega\right)$ with
the set of all continuous anti-linear functionals of $H_{0}^{1}\left(\Omega\right)$.
The anti-dual space of $H^{1}\left(\Omega\right)$ is denoted by $H^{1}\left(\Omega\right)^{*}$
and the dual product of $u\in H^{1}\left(\Omega\right)^{*}$ and $v\in H^{1}\left(\Omega\right)$
is denoted by $\left\langle u,v\right\rangle _{H^{1}\left(\Omega\right)^{*}\times H^{1}\left(\Omega\right)}$.
Similarly we define $H^{-\frac{1}{2}}\left(\partial\Omega\right)$
as the anti-dual space of $H^{\frac{1}{2}}\left(\partial\Omega\right)$
and denote the dual product of $u\in H^{-\frac{1}{2}}\left(\partial\Omega\right)$
and $v\in H^{\frac{1}{2}}\left(\partial\Omega\right)$ as $\left\langle u,v\right\rangle _{H^{-\frac{1}{2}}\left(\partial\Omega\right)\times H^{\frac{1}{2}}\left(\partial\Omega\right)}$.
As the inclusion $H^{\frac{1}{2}}\left(\partial\Omega\right)\hookrightarrow L^{2}\left(\partial\Omega\right)$
is injective with dense range, we can define an injective map with dense range
$L^{2}\left(\partial\Omega\right)\hookrightarrow H^{-\frac{1}{2}}\left(\partial\Omega\right)$
as $\left\langle u,v\right\rangle _{H^{-\frac{1}{2}}\left(\partial\Omega\right)\times H^{\frac{1}{2}}\left(\partial\Omega\right)}:=\left(u,v\right)_{L^{2}\left(\partial\Omega\right)}$,
when $u\in L^{2}\left(\partial\Omega\right)$ and $v\in H^{\frac{1}{2}}\left(\partial\Omega\right)$.

If $E$ and $F$ are Banach spaces, $\mathcal{B}\left(E,F\right)$ is
the set of all continuous maps from $E$ to $F$ and $\mathcal{B}\left(E\right):=\mathcal{B}\left(E,E\right)$.
The space $C_{u}^{\alpha}\left(\left[t_{0},\infty\right[,E\right)$
is the space of all uniformly $\alpha$-Hölder continuous functions from $\left[t_{0},\infty\right[$
to $E$, that is, if $f\in C_{u}^{\alpha}\left(\left[t_{0},\infty\right[,E\right)$,
then there is $C>0$ such that 
\[
\left\Vert f\left(t\right)-f\left(s\right)\right\Vert _{E}\le C\left|t-s\right|^{\alpha},\,\forall t,s\ge t_{0}.
\]
Similarly $C_{u}^{1,\alpha}\left(\left[t_{0},\infty\right[,E\right)$ is the set of all $C^1$ functions $f$ such that $f$ and $\frac{df}{dt}$ belong to $C_{u}^{\alpha}\left(\left[t_{0},\infty\right[,E\right)$.

In order to study the Dirichlet-to-Neumann problem, we consider the
following forms $a_{t}:H^{1}\left(\Omega\right)\times H^{1}\left(\Omega\right)\to\mathbb{C}$,
$t\in\left[t_{0},\infty\right]$:
\begin{equation}\label{eq:forms}
\begin{array}{rcl}
a_{t}\left(u,v\right) &=& \int_{\Omega}\left(\sum_{i,j=1}^{n}a_{ij}\left(t,x\right)\partial_{x_{j}}u\left(x\right)\overline{\partial_{x_{i}}v\left(x\right)}+\sum_{j=1}^{n}b_{j}\left(t,x\right)\partial_{x_{j}}u\left(x\right)\overline{v\left(x\right)}\right .\\
&& +
\left .\sum_{j=1}^{n}c_{j}\left(t,x\right)u\left(x\right)\overline{\partial_{x_{j}}v\left(x\right)}+d\left(t,x\right)u\left(x\right)\overline{v\left(x\right)}\right)dx.
\end{array}
\end{equation}
These are the forms associated with the differential operators $P\left(t,x,D\right)$
defined in \eqref{eq:P(t,x,D)}. Below are the basic assumptions we shall use in this paper.

\begin{fact}
\label{fact:Assumptionsoftheform} The coefficients of the forms $\left\{ a_{t}\right\} _{t\in\left[t_{0},\infty\right]}$ satisfy the \textbf{assumptions for $H^{-\frac{1}{2}}(\partial\Omega)$} if:

1) There is an $\alpha\in]0,1]$ such that $a_{ij},\,b_{j},\,c_{j},\hbox{ and }d\in C_{u}^{\alpha}\left(\left[t_{0},\infty\right[,L^{\infty}\left(\Omega\right)\right)$, for all $i$, $j$.

2) $a_{ij}\left(\infty,.\right),\,b_{j}\left(\infty,.\right),\,c_{j}\left(\infty,.\right)\hbox{ and }d\left(\infty,.\right)\in L^{\infty}\left(\Omega\right)$.

3) There is a constant $C>0$ such that
\begin{equation}
Re\left(a_{t}\left(u,u\right)\right)\ge C\left\Vert u\right\Vert _{H^{1}\left(\Omega\right)}^{2},\,\forall t\in\left[t_{0},\infty\right],\,\forall u\in H^{1}\left(\Omega\right).\label{eq:Rea(u,u)geu}
\end{equation}

4) $\lim_{t\to\infty}a_{ij}\left(t,.\right)=a_{ij}\left(\infty,.\right)$, $\lim_{t\to\infty}b_{j}\left(t,.\right)=b_{j}\left(\infty,.\right)$, $\lim_{t\to\infty}c_{j}\left(t,.\right)=c_{j}\left(\infty,.\right)$ and $\lim_{t\to\infty}d\left(t,.\right)=d\left(\infty,.\right)$ in $L^{\infty}\left(\Omega\right)$, for all $i$, $j$.
\end{fact}

\begin{fact}
\label{fact:Assumptionsoftheform2} The coefficients of the forms $\left\{ a_{t}\right\} _{t\in\left[t_{0},\infty\right]}$ satisfy the \textbf{assumptions for $L^{2}(\partial\Omega)$} if they satisfy all conditions of Assumption \ref{fact:Assumptionsoftheform} for $\alpha\in\left]\frac{1}{2},1\right]$.
\end{fact}

We will always assume that, at least, the Assumption 1 holds. The stronger Assumption 2 will only be necessary when we deal with the problem on $L^2(\partial\Omega)$, as it will be the case in Sections \ref{sec:2.1.4} and \ref{sec:3.2}. 

\begin{rem}\label{rem:M}
Conditions 1 and 4 of the Assumption 1 imply that the coefficients are bounded. Hence there is a
constant $M>0$ such that
\[
\left|a_{t}\left(u,v\right)\right|\le M\left\Vert u\right\Vert _{H^{1}\left(\Omega\right)}\left\Vert v\right\Vert _{H^{1}\left(\Omega\right)},\,\forall t\in\left[t_{0},\infty\right],\,u,v\in H^{1}\left(\Omega\right).
\]

The above conditions together with the Lax-Milgram Theorem can be used to define two important operators: $\mathcal{B}_{t,D}:H_{0}^{1}\left(\Omega\right)\to H^{-1}\left(\Omega\right)$
and $\mathcal{B}_{t,N}:H^{1}\left(\Omega\right)\to H^{1}\left(\Omega\right)^{*}$. They are the unique isomorphisms that satisfy
$$
a_{t}\left(u,v\right)=\left\langle \mathcal{B}_{t,D}\left(u\right),v\right\rangle _{H^{-1}\left(\Omega\right)\times H_{0}^{1}\left(\Omega\right),}
\quad \textrm{for all } u,v\in H_{0}^{1}\left(\Omega\right),
$$ 
and 
$$
a_{t}\left(u,v\right)=\left\langle \mathcal{B}_{t,N}\left(u\right),v\right\rangle _{H^{1}\left(\Omega\right)^{*}\times H^{1}\left(\Omega\right),}
\quad \textrm{for all } u,v\in H^{1}\left(\Omega\right). 
$$  
It is easy to see that $\mathcal{B}_{t,D}$ is equal to the operator $P\left(t,x,D\right)$ acting on $H_{0}^{1}\left(\Omega\right)$ in the sense of distributions and that, for all $t\in\left[t_{0},\infty\right]$:
\[
\left\Vert \mathcal{B}_{t,D}\right\Vert _{\mathcal{B}\left(H_{0}^{1}\left(\Omega\right),H^{-1}\left(\Omega\right)\right)}\le M\,,\left\Vert \mathcal{B}_{t,N}\right\Vert _{\mathcal{B}\left(H^{1}\left(\Omega\right),H^{1}\left(\Omega\right)^{*}\right)}\le M
\]
\[
\left\Vert \mathcal{B}_{t,D}^{-1}\right\Vert _{\mathcal{B}\left(H^{-1}\left(\Omega\right),H_{0}^{1}\left(\Omega\right)\right)}\le\frac{1}{C}\,\text{ and }\left\Vert \mathcal{B}_{t,N}^{-1}\right\Vert _{\mathcal{B}\left(H^{1}\left(\Omega\right)^{*},H^{1}\left(\Omega\right)\right)}\le\frac{1}{C}.
\]
\end{rem}

\begin{defn}
\label{def:conormal}
Let $u\in H^{1}\left(\Omega\right)$. We say that $C\left(t,x,D\right)u$
exists in the $H^{-\frac{1}{2}}\left(\partial\Omega\right)$-weak
sense and it is equal to $y\in H^{-\frac{1}{2}}\left(\partial\Omega\right)$
if 
\[
a_{t}\left(u,v\right)=\left\langle y,\gamma_{0}\left(v\right)\right\rangle _{H^{-\frac{1}{2}}\left(\partial\Omega\right)\times H^{\frac{1}{2}}\left(\partial\Omega\right)},\,\forall v\in H^{1}\left(\Omega\right).
\]
 If $y\in L^{2}\left(\partial\Omega\right)$, then we say that $C\left(t,x,D\right)u$
exists in the $L^{2}\left(\partial\Omega\right)$-weak sense.
\end{defn}

By the divergence theorem, the above definition coincides with the
usual conormal derivative if $P(t,x,D)u=0$ and if we impose sufficiently regularity to $u$ and to the coefficients. Our definition of weak conormal derivative can be found in a similar way in \cite{Arendtelst,AbrelJamilErica,ElstOuhabaz}.

If $u\in H^{1}\left(\Omega\right)$ is such that $C\left(t,x,D\right)u$
exists in the $H^{-\frac{1}{2}}\left(\partial\Omega\right)$-weak
sense, then considering $v\in C_{c}^{\infty}(\Omega)$ in the Definition
\ref{def:conormal}, we conclude that $P\left(t,x,D\right)u=0$.
On the other hand, we have:

\begin{prop}
Let $u\in H^{1}\left(\Omega\right)$ be such that $P\left(t,x,D\right)u=0$.
Then $C\left(t,x,D\right)u$ exists in the $H^{-\frac{1}{2}}\left(\partial\Omega\right)$-weak
sense. Moreover, there exists a constant $C>0$, which does not depend on $t\in[t_0,\infty[$, such that
\[
\left\Vert C\left(t,x,D\right)u\right\Vert _{H^{-\frac{1}{2}}\left(\partial\Omega\right)}\le C\left\Vert u\right\Vert _{H^{1}\left(\Omega\right)},
\]
whenever $t\in\left[t_{0},\infty\right]$ and $u\in H^{1}\left(\Omega\right)$ is such that $P\left(t,x,D\right)u=0$.
\end{prop}

\begin{proof}
Let $u\in H^{1}\left(\Omega\right)$ be such that $P\left(t,x,D\right)u=0$.
We define $y\in H^{-\frac{1}{2}}\left(\partial\Omega\right)$ as
\begin{equation}
\left\langle y,z\right\rangle _{H^{-\frac{1}{2}}\left(\partial\Omega\right)\times H^{\frac{1}{2}}\left(\partial\Omega\right)}=a_{t}\left(u,\mathcal{E}\left(z\right)\right), z\in H^\frac{1}{2}\left(\partial\Omega\right).\label{eq:yza-uez}
\end{equation}

First we show that the above definition is independent of the extension of $z$ we choose: if
$\tilde{z}\in H^{1}\left(\Omega\right)$ is such that $\gamma_{0}\left(\tilde{z}\right)=z$,
then 
\begin{equation}
\left\langle y,z\right\rangle _{H^{-\frac{1}{2}}\left(\partial\Omega\right)\times H^{\frac{1}{2}}\left(\partial\Omega\right)}=a_{t}\left(u,\tilde{z}\right).\label{eq:indoendenceofz}
\end{equation}
Indeed, by the construction of $y$, the expression (\ref{eq:indoendenceofz})
holds if $\tilde{z}=\mathcal{E}\left(z\right)$. Let us now suppose
that $\tilde{z}\in H^{1}\left(\Omega\right)$ is any other function
such that $\gamma_{0}\left(\tilde{z}\right)=z$. As $C_{c}^{\infty}\left(\Omega\right)$
is dense in $H_{0}^{1}\left(\Omega\right)$, we know that 
\[
a_{t}\left(v,w\right)=\left\langle P\left(t,x,D\right)v,w\right\rangle _{H^{-1}\left(\Omega\right)\times H_{0}^{1}\left(\Omega\right)},\,\forall v\in H^{1}\left(\Omega\right)\,\text{ and }\,w\in H_{0}^{1}\left(\Omega\right).
\]
Since $\mathcal{E}\left(z\right)-\tilde{z}\in H_{0}^{1}\left(\Omega\right)$
and $P\left(t,x,D\right)u=0$, we have 
\[
a_{t}\left(u,\mathcal{E}\left(z\right)-\tilde{z}\right)=\left\langle P\left(t,x,D\right)u,\mathcal{E}\left(z\right)-\tilde{z}\right\rangle _{H^{-1}\left(\Omega\right)\times H_{0}^{1}\left(\Omega\right)}=0.
\]
This implies that $\left\langle y,z\right\rangle _{H^{-\frac{1}{2}}\left(\partial\Omega\right)\times H^{\frac{1}{2}}\left(\partial\Omega\right)}=a_{t}\left(u,\mathcal{E}\left(z\right)\right)=a_{t}\left(u,\tilde{z}\right)$.
In particular, 
\[
a_{t}\left(u,v\right)=\left\langle y,\gamma_{0}\left(v\right)\right\rangle _{H^{-\frac{1}{2}}\left(\partial\Omega\right)\times H^{\frac{1}{2}}\left(\partial\Omega\right)},\,\forall v\in H^{1}\left(\Omega\right).
\]
Therefore $C\left(t,x,D\right)u$ exists in the $H^{-\frac{1}{2}}\left(\partial\Omega\right)$-sense and it is equal to $y$.

Finally note that (\ref{eq:yza-uez}) implies that 
\[\left\Vert C\left(t,x,D\right)u\right\Vert _{H^{-\frac{1}{2}}\left(\partial\Omega\right)}=\left\Vert y\right\Vert _{H^{-\frac{1}{2}}\left(\partial\Omega\right)}\le M\left\Vert \mathcal{E}\right\Vert _{\mathcal{B}\left(H^{\frac{1}{2}}\left(\partial\Omega\right),H^{1}\left(\Omega\right)\right)}\left\Vert u\right\Vert _{H^{1}\left(\Omega\right)},\]
where $M$ is the constant of Remark \ref{rem:M}.
\end{proof}
The definition of the Dirichlet-to-Neumann operator and the study of its asymptotic
behavior require the next simple proposition.
\begin{prop}
\label{lem:Dirichlet-Problem}
1) (Dirichlet Problem) Let $y\in H^{\frac{1}{2}}\left(\partial\Omega\right)$.
Then there is a unique $u_{t}\in H^{1}\left(\Omega\right)$ such that
$P\left(t,x,D\right)u_{t}=0$ and $\gamma_{0}\left(u_{t}\right)=y$.
It is given by $\mathcal{E}\left(y\right)-\mathcal{B}_{t,D}^{-1}P\left(t,x,D\right)\mathcal{E}\left(y\right)$.

2) (Neumann Problem) Let $y\in H^{-\frac{1}{2}}\left(\partial\Omega\right)$.
Then there is a unique $u_{t}\in H^{1}\left(\Omega\right)$ such that
$P\left(t,x,D\right)u_{t}=0$ and $C\left(t,x,D\right)u_{t}=y$. It
is given by $\mathcal{B}_{t,N}^{-1}\circ k\left(y\right)$, where
$k:H^{-\frac{1}{2}}\left(\partial\Omega\right)\to H^{1}\left(\Omega\right)^{*}$
is defined as 
\[
\left\langle k\left(w\right),v\right\rangle _{H^{1}\left(\Omega\right)^{*}\times H^{1}\left(\Omega\right)}=\left\langle w,\gamma_{0}\left(v\right)\right\rangle _{H^{-\frac{1}{2}}\left(\partial\Omega\right)\times H^{\frac{1}{2}}\left(\partial\Omega\right)}, v\in H^{1}\left(\Omega\right).
\]
\end{prop}

\begin{proof}
1) $u_{t}=\mathcal{E}\left(y\right)-\mathcal{B}_{t,D}^{-1}P\left(t,x,D\right)\mathcal{E}\left(y\right)$
is clearly a solution of the Dirichlet problem. If $v_{t}\in H^{1}\left(\Omega\right)$
is another solution, then $\gamma_{0}\left(u_{t}-v_{t}\right)=0$.
Hence $u_{t}-v_{t}\in H_{0}^{1}\left(\Omega\right)$ and, due to item 3 of Assumption \ref{fact:Assumptionsoftheform},
\begin{eqnarray}
C\left\Vert u_{t}-v_{t}\right\Vert _{H^{1}\left(\Omega\right)}^{2}&\le& \text{Re}\left(a_{t}\left(u_{t}-v_{t},u_{t}-v_{t}\right)\right)
\nonumber\\
&=&
\text{Re}\left(\left\langle P\left(t,x,D\right)\left(u_{t}-v_{t}\right),\left(u_{t}-v_{t}\right)\right\rangle _{H^{-1}\left(\Omega\right)\times H^{1}_{0}\left(\Omega\right)}\right)=0.
\nonumber
\end{eqnarray}

2) Let $u_{t}=\mathcal{B}_{t,N}^{-1}\circ k\left(y\right)$. Then, for $v\in H^1(\Omega)$,
\begin{equation}\label{eq:(*)}
a_{t}\left(\mathcal{B}_{t,N}^{-1}\circ k\left(y\right),v\right)=\left\langle \mathcal{B}_{t,N}\left(\mathcal{B}_{t,N}^{-1}\circ k\left(y\right)\right),v\right\rangle _{H^{1}\left(\Omega\right)^{*}\times H^{1}\left(\Omega\right)}=\left\langle y,\gamma_{0}\left(v\right)\right\rangle _{H^{-\frac{1}{2}}\left(\partial\Omega\right)\times H^{\frac{1}{2}}\left(\partial\Omega\right)}.
\end{equation}
We conclude that $y$ is the $H^{-\frac{1}{2}}\left(\partial\Omega\right)$-weak
conormal derivative of $u_{t}$ and, therefore, the solution of the
Neumann problem. If $v_{t}\in H^{1}\left(\Omega\right)$ is another
solution, then
\[
a_{t}\left(u_{t}-v_{t},u_{t}-v_{t}\right)=\left\langle C\left(t,x,D\right)\left(u_{t}-v_{t}\right),\gamma_{0}\left(u_{t}-v_{t}\right)\right\rangle _{H^{-\frac{1}{2}}\left(\partial\Omega\right)\times H^{\frac{1}{2}}\left(\partial\Omega\right)}=0.
\]

Hence $u_{t}=v_{t}$.
\end{proof}
Finally we give a precise definition of the Dirichlet-to-Neumann operator.
\begin{defn}\label{defn:DtN}
For each $t\in\left[t_{0},\infty\right]$, we define a bounded operator
$A\left(t\right):H^{\frac{1}{2}}\left(\partial\Omega\right)\to H^{-\frac{1}{2}}\left(\partial\Omega\right)$,
called Dirichlet-to-Neumann operator, as 
\[
A\left(t\right)y=C\left(t,x,D\right)\left(\mathcal{E}\left(y\right)-\mathcal{B}_{t,D}^{-1}P\left(t,x,D\right)\mathcal{E}\left(y\right)\right).
\]
\end{defn}

\begin{prop}\label{cor:uniforAt}
The operators $A\left(t\right)$ are invertible for all $t\in\left[t_{0},\infty\right]$.
Moreover the families 
\[\left\{ A\left(t\right)\in\mathcal{B}(H^{\frac{1}{2}}\left(\partial\Omega\right),H^{-\frac{1}{2}}\left(\partial\Omega\right))\right\} _{t\in\left[t_{0},\infty\right]}
\hbox{ and } \left\{ A\left(t\right)^{-1}\in\mathcal{B}(H^{-\frac{1}{2}}\left(\partial\Omega\right),H^{\frac{1}{2}}\left(\partial\Omega\right))\right\} _{t\in\left[t_{0},\infty\right]}
\]
are uniformly bounded.
\end{prop}

\begin{proof}
In order to conclude that the family $\left\{ A\left(t\right)\right\} _{t\in\left[t_{0},\infty\right]}$
is uniformly bounded, it is enough to note that $\mathcal{B}_{t,D}^{-1}:H^{-1}\left(\Omega\right)\to H_{0}^{1}\left(\Omega\right)$, $P\left(t,x,D\right):H^{1}\left(\Omega\right)\to H^{-1}\left(\Omega\right)$ and $C\left(t,x,D\right):\text{ker}\left(P\left(t,x,D\right)\right)\subset H^{1}\left(\Omega\right)\to H^{-\frac{1}{2}}\left(\partial\Omega\right)$ are uniformly bounded.

For the family $\left\{ A\left(t\right)^{-1}\right\} _{t\in\left[t_{0},\infty\right]}$,
we first need a good representation of the inverse. We have seen that
$A\left(t\right)y=C\left(t,x,D\right)u_{t}$, where $u_{t}=\mathcal{E}\left(y\right)-\mathcal{B}_{t,D}^{-1}\left(P\left(t,x,D\right)\mathcal{E}\left(y\right)\right)$.
Hence $u_{t}$ solves the Neumann problem $P\left(t,x,D\right)u_{t}=0$
and $C\left(t,x,D\right)u_{t}=A\left(t\right)y$. Proposition \ref{lem:Dirichlet-Problem} (2)
implies that $u_{t}=\mathcal{B}_{t,N}^{-1}\circ k\left(A\left(t\right)y\right)$.
Therefore $y=\gamma_{0}\left(u_{t}\right)=\gamma_{0}\circ\mathcal{B}_{t,N}^{-1}\circ k\left(A\left(t\right)\left(y\right)\right)$.

On the other hand, if $z\in H^{-\frac{1}{2}}\left(\partial\Omega\right)$, then, due to \eqref{eq:(*)}, $v_{t}=\mathcal{B}_{t,N}^{-1}\circ k\left(z\right)$ is such that 
\[
a_{t}\left(v_{t},v\right)=\left\langle z,\gamma_{0}\left(v\right)\right\rangle _{H^{-\frac{1}{2}}\left(\partial\Omega\right)\times H^{\frac{1}{2}}\left(\partial\Omega\right)},\,\forall v\in H^{1}\left(\Omega\right).
\]
Thus we conclude that $A\left(t\right)\left(\gamma_{0}\circ\mathcal{B}_{t,N}^{-1}\circ k\left(z\right)\right)=A\left(t\right)\left(\gamma_{0}\left(v_{t}\right)\right)=z$.

The above discussion implies that $A\left(t\right)^{-1}=\gamma_{0}\circ\mathcal{B}_{t,N}^{-1}\circ k$. Therefore $\left\{ A\left(t\right)^{-1}\right\} _{t\in\left[t_{0},\infty\right]}$ is uniformly bounded, since $\mathcal{B}_{t,N}^{-1}:H^1(\Omega)^{*}\to H^1(\Omega)$ is uniformly bounded.
\end{proof}

We end this subsection giving a characterization of the Dirichlet-to-Neumann operators from $H^\frac{1}{2}(\partial\Omega)$ to $H^{-\frac{1}{2}}(\partial\Omega)$ using form methods.

\begin{thm}
\label{thm:Existencia Dirichlet-to-Neumann} For every $y\in H^{\frac{1}{2}}\left(\partial\Omega\right)$,
there is a $u_{t}\in H^{1}\left(\Omega\right)$ such that
\begin{equation}
\gamma_{0}\left(u_{t}\right)=y\quad\textrm{ and }\quad a_{t}\left(u_{t},v\right)=\left\langle A\left(t\right)y,\gamma_{0}\left(v\right)\right\rangle _{H^{-\frac{1}{2}}\left(\partial\Omega\right)\times H^{\frac{1}{2}}\left(\partial\Omega\right)},\label{eq:characdeAt}
\end{equation}
 for all $v\in H^{1}\left(\Omega\right)$. The Dirichlet-to-Neumann operators $A\left(t\right)\in\mathcal{B}\left(H^\frac{1}{2}(\partial\Omega),H^{-\frac{1}{2}}(\partial\Omega)\right)$ are the only operators with this property.
\end{thm}

\begin{proof}
By definition, $A\left(t\right)y=C\left(t,x,D\right)u_{t}$, where $u_{t}=\mathcal{E}\left(y\right)-\mathcal{B}_{t,D}^{-1}P\left(t,x,D\right)\mathcal{E}\left(y\right)$. Using the definition of $C\left(t,x,D\right)u_{t}$, we see that
\[
a_{t}\left(u_{t},v\right)=\left\langle C\left(t,x,D\right)u_{t},\gamma_{0}\left(v\right)\right\rangle _{H^{-\frac{1}{2}}\left(\partial\Omega\right)\times H^{\frac{1}{2}}\left(\partial\Omega\right)},\,v\in H^{1}\left(\Omega\right),
\]
which is equivalent to (\ref{eq:characdeAt}).

Let us now prove uniqueness. Suppose that $A\left(t\right):H^{\frac{1}{2}}\left(\partial\Omega\right)\to H^{-\frac{1}{2}}\left(\partial\Omega\right)$
and $\tilde{A}\left(t\right):H^{\frac{1}{2}}\left(\partial\Omega\right)\to H^{-\frac{1}{2}}\left(\partial\Omega\right)$
are operators that satisfy the properties stated in the theorem. Then, for every $y\in H^{\frac{1}{2}}\left(\partial\Omega\right)$,
there exist $u_{t}\in H^{1}\left(\Omega\right)$ and $\tilde{u}_{t}\in H^{1}\left(\Omega\right)$
such that $\gamma_{0}\left(u_{t}\right)=\gamma_{0}\left(\tilde{u}_{t}\right)=y$,
and, for all $v\in H^{1}\left(\Omega\right)$, 
\[
a_{t}\left(u_{t},v\right)=\left\langle A\left(t\right)y,\gamma_{0}\left(v\right)\right\rangle _{H^{-\frac{1}{2}}\left(\partial\Omega\right)\times H^{\frac{1}{2}}\left(\partial\Omega\right)}\;\textrm{ and }\;a_{t}\left(\tilde{u}_{t},v\right)=\left\langle \tilde{A}\left(t\right)y,\gamma_{0}\left(v\right)\right\rangle _{H^{-\frac{1}{2}}\left(\partial\Omega\right)\times H^{\frac{1}{2}}\left(\partial\Omega\right)}.
\]
 In this case, 
\[
a_{t}\left(u_{t}-\tilde{u}_{t},v\right)=\left\langle A\left(t\right)y-\tilde{A}\left(t\right)y,\gamma_{0}\left(v\right)\right\rangle _{H^{-\frac{1}{2}}\left(\partial\Omega\right)\times H^{\frac{1}{2}}\left(\partial\Omega\right)},\,\forall v\in H^{1}\left(\Omega\right).
\]

Choosing $v=u_{t}-\tilde{u}_{t}$, we have $a_{t}\left(u_{t}-\tilde{u}_{t},u_{t}-\tilde{u}_{t}\right)=0$.
Hence $u_{t}=\tilde{u}_{t}$ and
\[ \left\langle A\left(t\right)y-\tilde{A}\left(t\right)y,\gamma_{0}\left(v\right)\right\rangle _{H^{-\frac{1}{2}}\left(\partial\Omega\right)\times H^{\frac{1}{2}}\left(\partial\Omega\right)}=0, \forall v\in H^{1}\left(\Omega\right),
\]
which implies that $A\left(t\right)=\tilde{A}\left(t\right)$.
\end{proof}
\begin{rem}
\label{rem:juleu}The proof of Theorem \ref{thm:Existencia Dirichlet-to-Neumann}
implies that the function $u_{t}$ associated to $y$ is unique and it
is given by $\mathcal{E}\left(y\right)-\mathcal{B}_{t,D}^{-1}\left(P\left(t,x,D\right)\mathcal{E}\left(y\right)\right)$.
As $\mathcal{B}_{t,D}^{-1}:H^{-1}(\Omega)\to H^{1}_{0}(\Omega)$ and $P\left(t,x,D\right):H^{1}(\Omega)\to H^{-1}(\Omega)$ are uniformly
bounded, there is a constant $C>0$, which does not depend on $t\in\left[t_{0},\infty\right]$,
such that 
\[
\left\Vert u_{t}\right\Vert _{H^{1}\left(\Omega\right)}\le C\left\Vert y\right\Vert _{H^{\frac{1}{2}}\left(\partial\Omega\right)}=C\left\Vert \gamma_{0}\left(u_{t}\right)\right\Vert _{H^{\frac{1}{2}}\left(\partial\Omega\right)},\,\forall y\in H^{\frac{1}{2}}\left(\partial\Omega\right).
\]
\end{rem}

\subsection{Well-posedness and asymptotic behavior.}

\subsubsection{The Tanabe-Sobolevskii conditions}

Let $H$ and $\mathcal{D}$ be Hilbert spaces such that $\mathcal{D}\subset H$
is a dense subset and the injection $\mathcal{D}\hookrightarrow H$ is
continuous. We consider a family of bounded operators $\left\{ S\left(t\right)\in\mathcal{B}\left(\mathcal{D},H\right)\right\} _{t\in\left[t_{0},\infty\right]}$.
\begin{defn}
\label{def:Tanabe-Sobolevskii conditions}The family of operators
$\left\{ S\left(t\right)\right\} _{t\in\left[t_{0},\infty\right]}$ satisfies
the Tanabe-Sobolevskii conditions if

1) The set $\left\{ \lambda\in\mathbb{C},\,\text{Re}\left(\lambda\right)\le0\right\} $
is contained in the resolvent set of the linear operator $S(t):\mathcal{D}\subset H\to H$,
$t\in\left[t_{0},\infty\right]$, and there is a constant $C>0$ such
that 
\[
\left\Vert \left(\lambda-S\left(t\right)\right)^{-1}\right\Vert _{\mathcal{B}\left(H\right)}\le\frac{C}{1+\left|\lambda\right|},\,\text{Re}\left(\lambda\right)\le0,\,t\in\left[t_{0},\infty\right[.
\]

2) The function $\left[t_{0},\infty\right[\ni t\mapsto S\left(t\right)\in\mathcal{B}\left(\mathcal{D},H\right)$ belongs to $C_{u}^{\alpha}\left(\left[t_{0},\infty\right[,\mathcal{B}\left(\mathcal{D},H\right)\right)$, for some $\alpha\in \left] 0,1 \right]$.


3) $\lim_{t\to\infty}\left\Vert S\left(t\right)-S\left(\infty\right)\right\Vert _{\mathcal{B}\left(\mathcal{D},H\right)}=0.$

4) The families $\left\{ S\left(t\right)\in\mathcal{B}\left(\mathcal{D},H\right)\right\} _{t\in\left[t_{0},\infty\right[}$
and $\left\{ S\left(t\right)^{-1}\in\mathcal{B}\left(H,\mathcal{D}\right)\right\} _{t\in\left[t_{0},\infty\right[}$
are uniformly bounded.
\end{defn}
\begin{thm}\label{thm:teorema10}
(Tanabe-Sobolevskii) Let $f\in C_{u}^{\alpha}\left(\left[t_{0},\infty\right[,H\right)$.
Then, for every $u_{0}\in H$, there is a unique function $u\in C\left(\left[t_{0},\infty\right[,H\right)\cap C^{1}\left(\left]t_{0},\infty\right[,H\right)\cap C\left(\left]t_{0},\infty\right[,\mathcal{D}\right)$
such that 

\begin{equation}
\begin{array}{rcl}
\frac{du}{dt}\left(t\right)+S\left(t\right)u\left(t\right) & = & f\left(t\right),\,t>t_{0}\\
u\left(t_{0}\right) & = & u_{0}
\end{array}.
\end{equation}

The operator $S\left(\infty\right):\mathcal{D}\to H$ is invertible and if $\lim_{t\to\infty}f\left(t\right)=f_{\infty}\in H$, then $u_{\infty}=S\left(\infty\right)^{-1}f_{\infty}\in\mathcal{D}$ is such that $\lim_{t\to\infty}\left\Vert u\left(t\right)-u_{\infty}\right\Vert _{\mathcal{D}}=0$. In other words, $u\left(t\right)$ converges to the stationary solution $S\left(\infty\right)u_{\infty}=f_{\infty}$.
\end{thm}
\begin{proof}
The existence and uniqueness of $u$ follows from Theorem 6.8 in \cite[Chapter
5.6]{Pazy}. One can even show that the solution is Hölder continuous
\cite[Theorem 1.2.1]{Amann}.

From Tanabe \cite[Theorem 5.6.1]{Tanabe} (see
also A. Pazy \cite{Pazy}), we know that $S\left(\infty\right):\mathcal{D}\to H$
is a bijective operator and, for $u_{\infty}=S\left(\infty\right)^{-1}f_{\infty}\in\mathcal{D}$, we have 
\begin{equation}
\lim_{t\to\infty}\left\Vert u\left(t\right)-u_{\infty}\right\Vert _{H}=0\,\,\,\text{and}\,\,\,\lim_{t\to\infty}\left\Vert \frac{du}{dt}\left(t\right)\right\Vert _{H}=0.\label{eq:u(t)-uinfty}
\end{equation}

As $\frac{du}{dt}\left(t\right)+S\left(t\right)u\left(t\right)=f\left(t\right)$,
we conclude that $\lim_{t\to\infty}\left\Vert u\left(t\right)-u_{\infty}\right\Vert _{\mathcal{D}}=0$.
In fact, we have that 
\begin{eqnarray*}
  \left\Vert u\left(t\right)-u_{\infty}\right\Vert _{\mathcal{D}}
 & = & \left\Vert S\left(t\right)^{-1}f\left(t\right)-S\left(t\right)^{-1}\frac{du}{dt}\left(t\right)-S\left(\infty\right)^{-1}f_{\infty}\right\Vert _{\mathcal{D}}\\
 & \le & \left\Vert S\left(t\right)^{-1}\right\Vert _{\mathcal{B}\left(H,\mathcal{D}\right)}\left\Vert f\left(t\right)-f_{\infty}\right\Vert _{H}+\left\Vert S\left(t\right)^{-1}-S\left(\infty\right)^{-1}\right\Vert _{\mathcal{B}\left(H,\mathcal{D}\right)}\left\Vert f_{\infty}\right\Vert _{H}\\
 &  & +\left\Vert S\left(t\right)^{-1}\right\Vert _{\mathcal{B}\left(H,\mathcal{D}\right)}\left\Vert \frac{du}{dt}\left(t\right)\right\Vert _{H}\to0.
\end{eqnarray*}

Note that the first and last terms on the right hand side of the above inequality
go to zero due to \eqref{eq:u(t)-uinfty}, to the convergence of the functions $f(t)$ and to the uniform boundedness of the set $\left\{S(t)^{-1}\right\}_{t\in \left[0,\infty\right[}$. Also the second one goes to zero,
due to the third and forth items of Definition \ref{def:Tanabe-Sobolevskii conditions}
and the inequality below
\[
\left\Vert S\left(t\right)^{-1}-S\left(\infty\right)^{-1}\right\Vert _{\mathcal{B}\left(H,\mathcal{D}\right)}\le\left\Vert S\left(t\right)^{-1}\right\Vert _{\mathcal{B}\left(H,\mathcal{D}\right)}\left\Vert S\left(t\right)-S\left(\infty\right)\right\Vert _{\mathcal{B}\left(\mathcal{D},H\right)}\left\Vert S\left(\infty\right)^{-1}\right\Vert _{\mathcal{B}\left(H,\mathcal{D}\right)}.
\]
\end{proof}

\subsubsection{The Dirichlet-to-Neumann operator in $H^{-\frac{1}{2}}\left(\partial\Omega\right)$}

The scalar product of $L^2(\partial\Omega)$ allows the definition of the map $y\in H^{\frac{1}{2}}\left(\partial\Omega\right)\mapsto\left(x\in H^{\frac{1}{2}}\left(\partial\Omega\right)\mapsto\left(y,x\right)_{L^{2}\left(\partial\Omega\right)}\in \mathbb{C}\right)\in H^{-\frac{1}{2}}\left(\partial\Omega\right)$. Using this map, we can identify $H^{\frac{1}{2}}\left(\partial\Omega\right)$ as a dense subspace of $H^{-\frac{1}{2}}\left(\partial\Omega\right)$. As always, we assume that Assumption 1 holds.
\begin{thm}\label{thm:teorema11}
The family  $\left\{ A\left(t\right)\in\mathcal{B}\left(H^{\frac{1}{2}}\left(\partial\Omega\right),H^{-\frac{1}{2}}\left(\partial\Omega\right)\right)\right\} _{t\in\left[t_{0},\infty\right]}$ of Dirichlet-to-Neumann operators defined by Definition \ref{defn:DtN} satisfies the Tanabe-Sobolevskii conditions. 
\end{thm}
\begin{proof}
Let us check all conditions of Definition \ref{def:Tanabe-Sobolevskii conditions}.

1. We define the form $a_{t,\lambda}:H^{1}\left(\Omega\right)\times H^{1}\left(\Omega\right)\to\mathbb{C}$
by 
\[a_{t,\lambda}\left(u,v\right)=a_{t}\left(u,v\right)-\lambda\left\langle \gamma_{0}\left(u\right),\gamma_{0}\left(v\right)\right\rangle _{H^{-\frac{1}{2}}\left(\partial\Omega\right)\times H^{\frac{1}{2}}\left(\partial\Omega\right)}.
\]
The form $a_{t,\lambda}$ is continuous and, as $Re\left(\lambda\right)\le0$, it satisfies $\text{Re}\left(a_{t,\lambda}\left(u,u\right)\right)\ge C\left\Vert u\right\Vert _{H^{1}\left(\Omega\right)}^{2}$.
Hence, by the Lax-Milgram Theorem, there is an isometry $\mathcal{B}_{t,\lambda,N}:H^{1}\left(\Omega\right)\to H^{1}\left(\Omega\right)^{*}$
such that
\[
a_{t,\lambda}\left(u,v\right)=\left\langle \mathcal{B}_{t,\lambda,N}\left(u\right),v\right\rangle _{H^{1}\left(\Omega\right)^{*}\times H^{1}\left(\Omega\right)}, u,v\in H^1 (\Omega).
\]
Using this form, we conclude that $A\left(t\right)-\lambda$ is invertible
and that $\left(A\left(t\right)-\lambda\right)^{-1}=\gamma_{0}\circ\mathcal{B}_{t,\lambda,N}^{-1}\circ k$, where $k$ is the map defined in Proposition \ref{lem:Dirichlet-Problem}. In fact, using the characterization provided by Theorem \ref{thm:Existencia Dirichlet-to-Neumann}, we have
\begin{eqnarray*}
 &  & y\in H^{\frac{1}{2}}\left(\partial\Omega\right),\,\left(A\left(t\right)-\lambda\right)y=f\\
 & \iff & \exists u_{t}\in H^{1}\left(\Omega\right)\,s.t.\,\gamma_{0}\left(u_{t}\right)=y\,\text{and}\,a_{t,\lambda}\left(u_{t},v\right)=\left\langle f,\gamma_{0}\left(v\right)\right\rangle _{H^{-\frac{1}{2}}\left(\partial\Omega\right)\times H^{\frac{1}{2}}\left(\partial\Omega\right)},v\in H^{1}\left(\Omega\right).\\
 & \iff & \exists u_{t}\in H^{1}\left(\Omega\right)\,s.t.\,\gamma_{0}\left(u_{t}\right)=y\,\text{and}\,\mathcal{B}_{t,\lambda,N}\left(u_{t}\right)=k\left(f\right)\iff y=\gamma_{0}\circ\mathcal{B}_{t,\lambda,N}^{-1}\circ k\left(f\right).
\end{eqnarray*}

Now suppose that $\left(A\left(t\right)-\lambda\right)y=f$. Then there is a unique $u_{t}\in H^{1}\left(\Omega\right)$
such that $\gamma_{0}\left(u_{t}\right)=y$ and 
\begin{equation}
a_{t,\lambda}\left(u_{t},v\right)=\left\langle f,\gamma_{0}\left(v\right)\right\rangle _{H^{-\frac{1}{2}}\left(\partial\Omega\right)\times H^{\frac{1}{2}}\left(\partial\Omega\right)},\,v\in H^{1}\left(\Omega\right).\label{eq:Aminuslambdaequalf}
\end{equation}
Setting $v=u_{t}$ in Equation \eqref{eq:Aminuslambdaequalf} and recalling that $\text{Re}(\lambda)\le0$, we obtain that
\begin{equation}
\begin{array}{rcl}
C\|u_{t}\|_{H^{1}(\Omega)}^{2}&\le& \,\text{Re}\,a_{t}\left(u_{t},u_{t}\right)\le \,\text{Re}\left\langle f,\gamma_{0}\left(u_{t}\right)\right\rangle _{H^{-\frac{1}{2}}\left(\partial\Omega\right)\times H^{\frac{1}{2}}\left(\partial\Omega\right)}
\\
&\le& \left\Vert f\right\Vert _{H^{-\frac{1}{2}}\left(\partial\Omega\right)}\left\Vert \gamma_{0}\left(u_{t}\right)\right\Vert _{H^{\frac{1}{2}}\left(\partial\Omega\right)}.
\end{array}\label{eq620}
\end{equation}

Equation (\ref{eq620}) and the boundedness of $\gamma_{0}:H^{1}\left(\Omega\right)\to H^{\frac{1}{2}}\left(\partial\Omega\right)$
imply that
\begin{equation}
\left\Vert \gamma_{0}\left(u_{t}\right)\right\Vert _{H^{\frac{1}{2}}\left(\partial\Omega\right)}\le C\left\Vert \gamma_{0}\right\Vert _{\mathcal{B}\left(H^{1}\left(\Omega\right),H^{\frac{1}{2}}\left(\partial\Omega\right)\right)}^{2}\left\Vert f\right\Vert _{H^{-\frac{1}{2}}\left(\partial\Omega\right)}.\label{eq:julef}
\end{equation}

The Equations (\ref{eq:Aminuslambdaequalf}) and \eqref{eq620}
show us that, for all $z\in H^{\frac{1}{2}}\left(\partial\Omega\right)$, we have
\begin{eqnarray*}
 &  & \left(1+\left|\lambda\right|\right)\left|\left\langle \gamma_{0}\left(u_{t}\right),z\right\rangle _{H^{-\frac{1}{2}}\left(\partial\Omega\right)\times H^{\frac{1}{2}}\left(\partial\Omega\right)}\right|\\
 & \le & \left|a_{t}\left(u_{t},\mathcal{E}\left(z\right)\right)\right|+\left|\left\langle f,z\right\rangle _{H^{-\frac{1}{2}}\left(\partial\Omega\right)\times H^{\frac{1}{2}}\left(\partial\Omega\right)}\right|+\left|\left\langle \gamma_{0}\left(u_{t}\right),z\right\rangle _{H^{-\frac{1}{2}}\left(\partial\Omega\right)\times H^{\frac{1}{2}}\left(\partial\Omega\right)}\right|\\
 & \le & M\left\Vert u_{t}\right\Vert _{H^{1}\left(\Omega\right)}\left\Vert \mathcal{E}\left(z\right)\right\Vert _{H^{1}\left(\Omega\right)}+\left(\left \Vert f\right\Vert _{H^{-\frac{1}{2}}\left(\partial\Omega\right)}+\left\Vert \gamma_{0}\left(u_{t}\right)\right\Vert _{H^{-\frac{1}{2}}\left(\partial\Omega\right)}\right)\left\Vert z\right\Vert _{H^{\frac{1}{2}}\left(\partial\Omega\right)}.
\end{eqnarray*}

Using the Remark \ref{rem:juleu}, the boundedness of $\mathcal{E}$
and Equation (\ref{eq:julef}), we conclude that 
\begin{equation}
\left|\left\langle \gamma_{0}\left(u_{t}\right),z\right\rangle _{H^{-\frac{1}{2}}\left(\partial\Omega\right)\times H^{\frac{1}{2}}\left(\partial\Omega\right)}\right|\le\frac{C}{1+\left|\lambda\right|}\left\Vert f\right\Vert _{H^{-\frac{1}{2}}\left(\Omega\right)}\left\Vert z\right\Vert _{H^{\frac{1}{2}}\left(\partial\Omega\right)},\,\forall z\in H^{\frac{1}{2}}\left(\partial\Omega\right).\label{eq:julef/1+lambda}
\end{equation}

The above expression implies that 
\[
\left\Vert \left(A\left(t\right)-\lambda I\right)^{-1}f\right\Vert _{H^{-\frac{1}{2}}\left(\partial\Omega\right)}\le\frac{C}{1+\left|\lambda\right|}\left\Vert f\right\Vert _{H^{-\frac{1}{2}}\left(\partial\Omega\right)}.
\]

2. First, we prove that $\left\Vert \mathcal{B}_{s,N}-\mathcal{B}_{t,N}\right\Vert _{\mathcal{B}\left(H^{1}\left(\Omega\right),H^{1}\left(\Omega\right)^{*}\right)}\le C\left|t-s\right|^{\alpha}$.
Indeed, due to Assumptions \ref{fact:Assumptionsoftheform},
we have 
\[
\left|\left\langle \mathcal{B}_{s,N}\left(u\right)-\mathcal{B}_{t,N}\left(u\right),v\right\rangle _{H^{1}\left(\Omega\right)^{*}\times H^{1}\left(\Omega\right)}\right|=\left|a_{s}\left(u,v\right)-a_{t}\left(u,v\right)\right|\le C\left|t-s\right|^{\alpha}\left\Vert u\right\Vert _{H^{1}\left(\Omega\right)}\left\Vert v\right\Vert _{H^{1}\left(\Omega\right)}.
\]

Second, we show that $\left\Vert \mathcal{B}_{s,N}^{-1}-\mathcal{B}_{t,N}^{-1}\right\Vert _{\mathcal{B}\left(H^{1}\left(\Omega\right)^{*},H^{1}\left(\Omega\right)\right)}\le C\left|t-s\right|^{\alpha}$. This follows from:
\begin{eqnarray}
& &\left\Vert \mathcal{B}_{t,N}^{-1}-\mathcal{B}_{s,N}^{-1}\right\Vert _{\mathcal{B}\left(H^{1}\left(\Omega\right)^{*},H^{1}\left(\Omega\right)\right)}=\left\Vert \mathcal{B}_{s,N}^{-1}\left(\mathcal{B}_{s,N}-\mathcal{B}_{t,N}\right)\mathcal{B}_{t,N}^{-1}\right\Vert _{\mathcal{B}\left(H^{1}\left(\Omega\right)^{*},H^{1}\left(\Omega\right)\right)}
\nonumber\\
&\le&
\left\Vert \mathcal{B}_{s,N}^{-1}\right\Vert _{\mathcal{B}\left(H^{1}\left(\Omega\right)^{*},H^{1}\left(\Omega\right)\right)}\left\Vert \mathcal{B}_{s,N}-\mathcal{B}_{t,N}\right\Vert _{\mathcal{B}\left(H^{1}\left(\Omega\right),H^{1}\left(\Omega\right)^{*}\right)}\left\Vert \mathcal{B}_{t,N}^{-1}\right\Vert _{\mathcal{B}\left(H^{1}\left(\Omega\right)^{*},H^{1}\left(\Omega\right)\right)}.
\nonumber
\end{eqnarray}

Finally, the uniform boundedness of the family of operators $\left\{A\left(t\right)\in\mathcal{B}\left(H^{\frac{1}{2}}\left(\partial\Omega\right),H^{-\frac{1}{2}}\left(\partial\Omega\right)\right)\right\}_{t\in\left[t_{0},\infty\right]}$ and the fact that $A\left(t\right)^{-1}=\gamma_{0}\circ\mathcal{B}_{t,N}^{-1}\circ k$ imply the second condition of Definition \ref{def:Tanabe-Sobolevskii conditions}, due to

\[
\left\Vert A\left(t\right)-A\left(s\right)\right\Vert _{\mathcal{B}\left(H^{\frac{1}{2}}\left(\partial\Omega\right),H^{-\frac{1}{2}}\left(\partial\Omega\right)\right)}=\left\Vert A\left(s\right)\left(A\left(s\right)^{-1}-A\left(t\right)^{-1}\right)A\left(t\right)\right\Vert _{\mathcal{B}\left(H^{\frac{1}{2}}\left(\partial\Omega\right),H^{-\frac{1}{2}}\left(\partial\Omega\right)\right)}.
\]

3. The Assumption 1 implies that $\lim_{t\to\infty}\left(\sup_{\left\Vert u\right\Vert _{H^{1}\left(\Omega\right)}=\left\Vert v\right\Vert _{H^{1}\left(\Omega\right)}=1}\left|a_{t}\left(u,v\right)-a_{\infty}\left(u,v\right)\right|\right)=0$. The proof then follows the same arguments of the second item.

4. The forth condition follows from Proposition \ref{cor:uniforAt}.
\end{proof}
\begin{cor}\label{cor:dthhmeio}
Let $f\in C_{u}^{\alpha}\left(\left[t_{0},\infty\right[,H^{-\frac{1}{2}}\left(\partial\Omega\right)\right)$
and $A\left(t\right):H^{\frac{1}{2}}\left(\partial\Omega\right)\subset H^{-\frac{1}{2}}\left(\partial\Omega\right)\to H^{-\frac{1}{2}}\left(\partial\Omega\right)$
be the Dirichlet-to-Neumann operators associated with the forms $\left\{ a_{t}:H^{1}\left(\Omega\right)\times H^{1}\left(\Omega\right)\to\mathbb{C}\right\} _{t\in\left[t_{0},\infty\right]}$ defined by \eqref{eq:forms}
and satisfying Assumption \ref{fact:Assumptionsoftheform}. 

Then, for every $u_{0}\in H^{-\frac{1}{2}}\left(\partial\Omega\right)$,
there is a unique function $u\in C\left(\left[t_{0},\infty\right[,H^{-\frac{1}{2}}\left(\partial\Omega\right)\right)\cap C^{1}\left(\left]t_{0},\infty\right[,H^{-\frac{1}{2}}\left(\partial\Omega\right)\right)\cap C\left(\left]t_{0},\infty\right[,H^{\frac{1}{2}}\left(\partial\Omega\right)\right)$
that solves Problem \eqref{eq:mainproblem}.

Moreover, $A(\infty):H^{\frac{1}{2}}(\partial\Omega)\to H^{-\frac{1}{2}}(\partial\Omega)$ is invertible and if $\lim_{t\to\infty}f\left(t\right)=f_{\infty}\in H^{-\frac{1}{2}}\left(\partial\Omega\right)$,
then $\lim_{t\to\infty}\left\Vert u\left(t\right)-u_{\infty}\right\Vert _{H^{\frac{1}{2}}\left(\partial\Omega\right)}=0$, where $u_{\infty}$ is the unique solution of $A(\infty)u_{\infty}=f_{\infty}$. 
\end{cor}

\begin{proof}
It is a consequence of Theorem \ref{thm:teorema10} and \ref{thm:teorema11}.
\end{proof}

\subsubsection{The Yagi conditions}\label{sec:2.1.3}

It is natural to consider the operator Dirichlet-to-Neumann acting
on functions instead of distribution spaces. In order to study the Dirichlet-to-Neumann problem in $L^2(\partial\Omega)$, we will apply the results of A. Yagi \cite[Chapter 3]{Yagibook}. Let us recall them in this section. We fix a complex Hilbert space $H$.

\begin{defn}\label{def:Yagiconditions} We say that a family $\left\{ S\left(t\right):\mathcal{D}\left(S\left(t\right)\right)\subset H\to H\right\} _{t\in\left[t_{0},\infty\right]}$ of closed and densely defined operators  satisfies the Yagi conditions if there exist constants $M\ge1$,
$0<\nu\le1$, $0<\alpha\le1$ with $\alpha+\nu>1$, such that

1) The set $\left\{ \lambda\in\mathbb{C},\,\text{Re}\left(\lambda\right)\le0\right\} $
is contained in the resolvent set of the linear operator $S(t):\mathcal{D}(S(t))\subset H\to H$,
$t\in\left[t_{0},\infty\right]$, and
\[
\left\Vert \left(\lambda-S\left(t\right)\right)^{-1}\right\Vert _{\mathcal{B}\left(H\right)}\le\frac{M}{1+\left|\lambda\right|},\,\text{Re}\left(\lambda\right)\le0,\,t\in\left[t_{0},\infty\right[.
\]

2) For all $t,s\in\left[t_{0},\infty\right]$, we have $\mathcal{D}\left(S\left(s\right)\right)\subset\mathcal{D}\left(S\left(t\right)^{\nu}\right)$
and 

\[
\left\Vert S\left(t\right)^{\nu}\left(S\left(t\right)^{-1}-S\left(s\right)^{-1}\right)\right\Vert _{\mathcal{B}\left(H\right)}\le C\left|t-s\right|^{\alpha}.
\]
\end{defn}
\begin{rem}
If $\nu=1$, then the item 2) implies that the domains of the operators
are constant. Moreover, if $\left\Vert S\left(t\right)\right\Vert _{\mathcal{B}\left(H\right)}$
and $\left\Vert S\left(t\right)^{-1}\right\Vert _{\mathcal{B}\left(H\right)}$
are uniformily bounded, then it is equivalent to condition 2 of
Definition \ref{def:Tanabe-Sobolevskii conditions}.
\end{rem}

The following theorem of A. Yagi can be found in \cite[Theorem 3.9 on page 147]{Yagibook}.
\begin{thm}
(Yagi) Let $f\in C_{u}^{\alpha}\left(\left[t_{0},\infty\right[,H\right)$ and $\left\{ S\left(t\right):\mathcal{D}\left(S\left(t\right)\right)\subset H\to H\right\}$ be a set of operators that satisfy the Yagi conditions. Then, for every $u_{0}\in H$, there is a unique function $u\in C\left(\left[t_{0},\infty\right[,H\right)\cap C^{1}\left(\left]t_{0},\infty\right[,H\right)$
such that $u\left(t\right)\in\mathcal{D}\left(S\left(t\right)\right)$,
for all $t\in\left]t_{0},\infty\right[$, and 
\[
\begin{array}{c}
\frac{du}{dt}\left(t\right)+S\left(t\right)u\left(t\right)=f\left(t\right)\\
u\left(t_{0}\right)=u_{0}
\end{array}.
\]
\end{thm}

\subsubsection{The Dirichlet-to-Neumann operator in
$L^{2}\left(\partial\Omega\right)$.\label{subsec:dtnl2}}\label{sec:2.1.4}

In this section, we assume that the stronger Assumption 2 holds, that is, $\alpha\in \left]\frac{1}{2},1\right]$. Our aim is to study the Dirichlet-to-Neumann operator on $L^2(\Omega)$ according to W. Arendt and A. Elst \cite{Arendtelst,Arendtformmethods}.
For each $t\in\left[t_{0},\infty\right]$, we set 
$$\left.A\left(t\right)\right|_{L^{2}\left(\partial\Omega\right)}:\mathcal{D}\left(\left.A\left(t\right)\right|_{L^{2}\left(\partial\Omega\right)}\right)\subset L^{2}\left(\partial\Omega\right)\to L^{2}\left(\partial\Omega\right)$$
as the part of our previous operators $A\left(t\right)$
in $L^{2}\left(\partial\Omega\right)$. It is the operator
$A\left(t\right)$ acting on the domain
\[
\mathcal{D}\left(\left.A\left(t\right)\right|_{L^{2}\left(\partial\Omega\right)}\right)=\left\{ y\in H^{\frac{1}{2}}\left(\partial\Omega\right);\,A\left(t\right)y\in L^{2}\left(\partial\Omega\right)\right\} .
\]

The part of $A\left(t\right)$ in $H^{\frac{1}{2}}\left(\partial\Omega\right)$,
that is, the operator 
$$\left.A\left(t\right)\right|_{H^{\frac{1}{2}}\left(\partial\Omega\right)}:\mathcal{D}\left(\left.A\left(t\right)\right|_{H^{\frac{1}{2}}\left(\partial\Omega\right)}\right)\subset H^{\frac{1}{2}}\left(\partial\Omega\right)\to H^{\frac{1}{2}}\left(\partial\Omega\right)$$
defined as the operator $A\left(t\right)$ acting on 
\[
\mathcal{D}\left(\left.A\left(t\right)\right|_{H^{\frac{1}{2}}\left(\partial\Omega\right)}\right)=\left\{ y\in H^{\frac{1}{2}}\left(\partial\Omega\right),\,A\left(t\right)y\in H^{\frac{1}{2}}\left(\partial\Omega\right)\right\} ,
\]
is also useful in our analysis.

Notice that, for all $t\in\left[t_{0},\infty\right]$, the operators $\left.A\left(t\right)\right|_{L^{2}\left(\partial\Omega\right)}$
and $\left.A\left(t\right)\right|_{H^{\frac{1}{2}}\left(\partial\Omega\right)}$
are densely defined. In fact, $A\left(t\right)^{-1}:H^{-\frac{1}{2}}\left(\partial\Omega\right)\to H^{\frac{1}{2}}\left(\partial\Omega\right)$
is an isomorphism, $H^{\frac{1}{2}}\left(\partial\Omega\right)$ is
dense in $L^{2}\left(\partial\Omega\right)$ and in $H^{-\frac{1}{2}}\left(\partial\Omega\right)$, $\mathcal{D}\left(\left.A\left(t\right)\right|_{L^{2}\left(\partial\Omega\right)}\right)=A\left(t\right)^{-1}\left(L^{2}\left(\partial\Omega\right)\right)$
and $\mathcal{D}\left(\left.A\left(t\right)\right|_{H^{\frac{1}{2}}\left(\partial\Omega\right)}\right)=A\left(t\right)^{-1}\left(H^{\frac{1}{2}}\left(\partial\Omega\right)\right)$.
\begin{prop}
\label{prop:L2sectorial}For all $\text{Re}(\lambda)\le0$ and $t\in\left[t_{0},\infty\right]$,
the operators 
$$
\left.A\left(t\right)\right|_{L^{2}\left(\partial\Omega\right)}-\lambda
\quad  \textrm{ and } \quad 
\left.A\left(t\right)\right|_{H^{\frac{1}{2}}\left(\partial\Omega\right)}-\lambda
$$
are invertible. Moreover, there is a constant $C>0$ such that 
$$\left\Vert \left(\left.A\left(t\right)\right|_{L^{2}\left(\partial\Omega\right)}-\lambda\right)^{-1}\right\Vert _{\mathcal{B}\left(L^{2}\left(\partial\Omega\right)\right)}\le\frac{C}{1+\left|\lambda\right|}$$
and 
$$\left\Vert \left(\left.A\left(t\right)\right|_{H^{\frac{1}{2}}\left(\partial\Omega\right)}-\lambda\right)^{-1}\right\Vert _{\mathcal{B}\left(H^{\frac{1}{2}}\left(\partial\Omega\right)\right)}\le\frac{C}{1+\left|\lambda\right|}$$
for all $\text{Re}(\lambda)\le0$ and $t\in\left[t_{0},\infty\right]$.
\end{prop}
\begin{proof}
It is clear $\left.A\left(t\right)\right|_{L^{2}\left(\partial\Omega\right)}-\lambda$
and $\left.A\left(t\right)\right|_{H^{\frac{1}{2}}\left(\partial\Omega\right)}-\lambda$
are bijections, since $A\left(t\right)-\lambda I:H^{\frac{1}{2}}\left(\partial\Omega\right)\to H^{-\frac{1}{2}}\left(\partial\Omega\right)$
is also one.

Let us prove the inequality for $\left.A\left(t\right)\right|_{L^{2}\left(\partial\Omega\right)}$. We consider $\text{Re}\left(\lambda\right)\le0$, $y\in H^{\frac{1}{2}}\left(\partial\Omega\right)$
and $f\in L^{2}\left(\partial\Omega\right)$ such that $\left(A\left(t\right)-\lambda\right)y=f$.
Then there exists a $u_{t}\in H^{1}(\Omega)$ that satisfies Equation
\eqref{eq:Aminuslambdaequalf} and is such that $\gamma_{0}\left(u_{t}\right)=y$.
Consequently, we obtain that 
\begin{eqnarray} \label{eq:julelambdafL2}
 &  & \lefteqn{\left(1+\left|\lambda\right|\right)\left\Vert \gamma_{0}\left(u_{t}\right)\right\Vert _{L^{2}\left(\partial\Omega\right)}^{2}}\nonumber \\
 & \le & \left|a_{t}\left(u_{t},u_{t}\right)\right|+\left|\left\langle f,\gamma_{0}\left(u_{t}\right)\right\rangle _{H^{-\frac{1}{2}}\left(\partial\Omega\right)\times H^{\frac{1}{2}}\left(\partial\Omega\right)}\right|+\left|\left\langle \gamma_{0}\left(u_{t}\right),\gamma_{0}\left(u_{t}\right)\right\rangle _{H^{-\frac{1}{2}}\left(\partial\Omega\right)\times H^{\frac{1}{2}}\left(\partial\Omega\right)}\right| \\
 & \le & C\left(\left\Vert u_{t}\right\Vert _{H^{1}\left(\Omega\right)}^{2}+\left\Vert f\right\Vert _{L^{2}\left(\partial\Omega\right)}\left\Vert \gamma_{0}\left(u_{t}\right)\right\Vert _{L^{2}\left(\partial\Omega\right)}+\left\Vert \gamma_{0}\left(u_{t}\right)\right\Vert _{L^{2}\left(\partial\Omega\right)}^{2}\right). \nonumber
\end{eqnarray}
 Arguing as in Equation \eqref{eq620}, and noting
that, if $f\in L^{2}\left(\partial\Omega\right)$, then 
$$\left\langle f,\gamma_{0}\left(u_{t}\right)\right\rangle _{H^{-\frac{1}{2}}\left(\partial\Omega\right)\times H^{\frac{1}{2}}\left(\partial\Omega\right)}=\left(f,\gamma_{0}\left(u_{t}\right)\right)_{L^{2}\left(\partial\Omega\right)},$$
we get that  
\begin{equation}
\left\Vert u_{t}\right\Vert _{H^{1}\left(\Omega\right)}^{2}\le C\left\Vert f\right\Vert _{L^{2}\left(\partial\Omega\right)}\left\Vert \gamma_{0}\left(u_{t}\right)\right\Vert _{L^{2}\left(\partial\Omega\right)}.\label{eq:aaa}
\end{equation}
Now, using the continuity of the trace and Equation \eqref{eq:aaa}, we have that 
\begin{equation} \label{eq:bbb}
\begin{array}{rcl}
\left\Vert \gamma_{0}\left(u_{t}\right)\right\Vert _{L^{2}\left(\partial\Omega\right)}^{2}&\le&\left\Vert \gamma_{0}\left(u_{t}\right)\right\Vert _{H^{\frac{1}{2}}\left(\partial\Omega\right)}^{2} 
\le\left\Vert \gamma_{0}\right\Vert _{\mathcal{B}\left(H^{1}\left(\Omega\right),H^{\frac{1}{2}}\left(\partial\Omega\right)\right)}^{2}\left\Vert u_{t}\right\Vert _{H^{1}\left(\Omega\right)}^{2} \\ 
&\le& C\left\Vert f\right\Vert _{L^{2}\left(\partial\Omega\right)}\left\Vert \gamma_{0}\left(u_{t}\right)\right\Vert _{L^{2}\left(\partial\Omega\right)}.
\end{array}
\end{equation}
 Applying Equations \eqref{eq:aaa} and \eqref{eq:bbb}
to the Equation \eqref{eq:julelambdafL2}, we conclude that 
$$\left\Vert \gamma_{0}\left(u_{t}\right)\right\Vert _{L^{2}\left(\partial\Omega\right)}\le\frac{C}{1+\left|\lambda\right|}\left\Vert f\right\Vert _{L^{2}\left(\partial\Omega\right)}.$$

Finally, let us prove the inequality for $\left.A\left(t\right)\right|_{H^{\frac{1}{2}}\left(\partial\Omega\right)}$. We take $\text{Re}\left(\lambda\right)\le0$ and $f\in H^{\frac{1}{2}}\left(\partial\Omega\right)$. As 
\[
\lambda\left(\lambda-A\left(t\right)\right)^{-1}f=\left(\lambda-A\left(t\right)\right)^{-1}A\left(t\right)f+f,
\]
we conclude that
\begin{equation}\label{eq:uhhhh}
\begin{array}{rcl}
& & \left|\lambda\right|\left\Vert \left(\lambda-A\left(t\right)\right)^{-1}f\right\Vert _{H^{\frac{1}{2}}\left(\partial\Omega\right)}
\\
&=& \left(\left\Vert \left(\lambda-A\left(t\right)\right)^{-1}\right\Vert _{\mathcal{B}\left(H^{-\frac{1}{2}}\left(\partial\Omega\right),H^{\frac{1}{2}}\left(\partial\Omega\right)\right)}\left\Vert A\left(t\right)\right\Vert _{\mathcal{B}\left(H^{\frac{1}{2}}\left(\partial\Omega\right),H^{-\frac{1}{2}}\left(\partial\Omega\right)\right)}+1\right)\left\Vert f\right\Vert _{H^{\frac{1}{2}}\left(\partial\Omega\right)}
\\
&\le &  C\left\Vert f\right\Vert _{H^{\frac{1}{2}}\left(\partial\Omega\right)}.
\end{array}
\end{equation}

The last inequality was obtained using Equation \eqref{eq:julef} and the uniform boundedness of $A(t)$.
The inequality from the statement of the Theorem follows then from
Equation \eqref{eq:uhhhh} and the fact that $\lambda-A\left(t\right)$,
and therefore $\lambda-\left.A\left(t\right)\right|_{H^{\frac{1}{2}}\left(\partial\Omega\right)}$,
is invertible in a neighborhood of the origin.
\end{proof}
The Dirichlet-to-Neumann problem on $L^{2}\left(\partial\Omega\right)$ can now be defined by:
\begin{equation}
\begin{array}{rcl}
\frac{du}{dt}\left(t\right)+\left.A\left(t\right)\right|_{L^{2}\left(\partial\Omega\right)}u\left(t\right) & = & f\left(t\right),\,t>t_{0}\\
u\left(t_{0}\right) & = & u_{0},\label{eq770}
\end{array}
\end{equation}
 where $u_{0}\in L^{2}\left(\partial\Omega\right)$ and $f\in C_{u}^{\alpha}\left(\left[t_{0},\infty\right[,L^{2}\left(\partial\Omega\right)\right)$.

The well-posedness of the above problem is the main Theorem of this
section:
\begin{thm}
\label{thm:L2} Suppose that the forms \eqref{eq:forms} satisfy the Assumption 2, that is, $\alpha\in\left]\frac{1}{2},1\right]$, and let $f\in C_{u}^{\alpha}\left(\left[t_{0},\infty\right[,L^{2}\left(\partial\Omega\right)\right)$. Then there is a unique $u\in C\left(\left[t_{0},\infty\right[,L^{2}\left(\partial\Omega\right)\right)\cap C^{1}\left(\left]t_{0},\infty\right[,L^{2}\left(\partial\Omega\right)\right)\cap C\left(\left]t_{0},\infty\right[,H^{\frac{1}{2}}\left(\partial\Omega\right)\right)$
such that $u\left(t\right)\in\mathcal{D}\left(\left.A\left(t\right)\right|_{L^{2}\left(\partial\Omega\right)}\right)$,
for all $t>t_{0}$, and such that $u$ is a solution of the Problem
\eqref{eq770}.
\end{thm}
An elegant way to prove Theorem \ref{thm:L2} is due to A. Yagi.
In \cite{Yagibook,Yagipaper}, well-posedness of non-autonomous problems were obtained to operators
defined by forms in the traditional way, when the domain of the operator
is a subset of the domain of the form. In the case we are considering here, the
domain of the form is $H^{1}\left(\Omega\right)$, and the domain of the operator
is a set contained in $H^{\frac{1}{2}}\left(\partial\Omega\right)$.
One set is not even included in the other. It is clear that some changes are necessary.

In order to provide a full proof, we argue as Yagi. Lemma \ref{lem:lemma20} below is essentially contained in \cite[Theorem 2.32, page 110]{Yagibook}. The idea of proof of Theorem \ref{thm:Yagi} comes from \cite[page 149]{Yagibook}, although here it requires some results obtained in the previous sections of this paper. A different prove could be given using the methods of Tanabe \cite[Section 5.4]{Tanabe}, although it would also require some modifications and the same hypothesis $\alpha\in\left]\frac{1}{2},1\right]$.

Theorem \ref{thm:L2} is obtained as a consequence
of Corollary \ref{cor:dthhmeio} and the following result.
\begin{thm}
\label{thm:Yagi} If the Assumption 2 is fulfilled, then for all $\nu\in\left]1-\alpha,\frac{1}{2}\right[$, the family of operators 
$$\left\{ \left.A\left(t\right)\right|_{L^{2}\left(\partial\Omega\right)}:\mathcal{D}\left(\left.A\left(t\right)\right|_{L^{2}\left(\partial\Omega\right)}\right)\subset L^{2}\left(\partial\Omega\right)\to L^{2}\left(\partial\Omega\right)\right\} _{t\in\left[t_{0},\infty\right]}$$ 
satisfies the Yagi conditions.
\end{thm}
First, we fix some notation. Using the duality of $H^{\frac{1}{2}}\left(\partial\Omega\right)$
and $H^{-\frac{1}{2}}\left(\partial\Omega\right)$, for each
$A:H^{\frac{1}{2}}\left(\partial\Omega\right)\to H^{-\frac{1}{2}}\left(\partial\Omega\right)$,
we define $A^{*}:H^{\frac{1}{2}}\left(\partial\Omega\right)\to H^{-\frac{1}{2}}\left(\partial\Omega\right)$ as the operator such that
\[
\left\langle Au,v\right\rangle _{H^{-\frac{1}{2}}\left(\partial\Omega\right)\times H^{\frac{1}{2}}\left(\partial\Omega\right)}=\overline{\left\langle A^{*}v,u\right\rangle _{H^{-\frac{1}{2}}\left(\partial\Omega\right)\times H^{\frac{1}{2}}\left(\partial\Omega\right)}},\,\forall u,v\in H^{\frac{1}{2}}\left(\partial\Omega\right).
\]

If $B:H^{-\frac{1}{2}}\left(\partial\Omega\right)\to H^{\frac{1}{2}}\left(\partial\Omega\right)$,
we define $B^{*}:H^{-\frac{1}{2}}\left(\partial\Omega\right)\to H^{\frac{1}{2}}\left(\partial\Omega\right)$
as the operator such that
\[
\left\langle u,Bv\right\rangle _{H^{-\frac{1}{2}}\left(\partial\Omega\right)\times H^{\frac{1}{2}}\left(\partial\Omega\right)}=\overline{\left\langle v,B^{*}u\right\rangle _{H^{-\frac{1}{2}}\left(\partial\Omega\right)\times H^{\frac{1}{2}}\left(\partial\Omega\right)}},\,\forall u,v\in H^{-\frac{1}{2}}\left(\partial\Omega\right).
\]

By the above definitions, it is clear that $A^{*}$ and $B^{*}$ are uniquely defined and $\left(A^{*}\right)^{-1}=\left(A^{-1}\right)^{*}$. Moreover, if $A$ is the Dirichlet-to-Neumann operator associated with the form $a:H^1(\Omega)\times H^1(\Omega)\to \mathbb{C}$, then $A^{*}$ is operator associated with $a^{*}:H^1(\Omega)\times H^1(\Omega)\to \mathbb{C}$ defined by $a^{*}\left(u,v\right)=\overline{a\left(v,u\right)}$.
\begin{lem}\label{lem:lemma19}
Let $A:H^{\frac{1}{2}}\left(\partial\Omega\right)\subset H^{-\frac{1}{2}}\left(\partial\Omega\right)\to H^{-\frac{1}{2}}\left(\partial\Omega\right)$ be an operator such that $\text{Re}\left(\lambda\right)\le0$ is contained in the resolvent set and $\left\Vert \left(\lambda-A\right)^{-1}\right\Vert \le\frac{C}{1+\left|\lambda\right|}$. Then, if $u\in H^{-\frac{1}{2}}\left(\partial\Omega\right)$
and $v\in H^{\frac{1}{2}}\left(\partial\Omega\right)$, we have $\left\langle A^{-\theta}u,v\right\rangle _{H^{-\frac{1}{2}}\left(\partial\Omega\right)\times H^{\frac{1}{2}}\left(\partial\Omega\right)}=\left\langle u,\left(A^{*}\right)^{-\theta}v\right\rangle _{H^{-\frac{1}{2}}\left(\partial\Omega\right)\times H^{\frac{1}{2}}\left(\partial\Omega\right)}$ for all $\theta\in\left]0,1\right[$.
\end{lem}
\begin{proof}
It follows from the definitions given. In fact,
\begin{eqnarray}
& & \frac{\pi}{\sin\left(\theta\pi\right)}\left\langle A^{-\theta}u,v\right\rangle _{H^{-\frac{1}{2}}\left(\partial\Omega\right)\times H^{\frac{1}{2}}\left(\partial\Omega\right)}=\int_{0}^{\infty}\rho^{-\theta}\left\langle \left(\rho+A\right)^{-1}u,v\right\rangle _{H^{-\frac{1}{2}}\left(\partial\Omega\right)\times H^{\frac{1}{2}}\left(\partial\Omega\right)}d\rho
\nonumber\\
&=& \int_{0}^{\infty}\rho^{-\theta}\overline{\left\langle v,\left(\left(\rho+A\right)^{-1}\right)u\right\rangle _{H^{-\frac{1}{2}}\left(\partial\Omega\right)\times H^{\frac{1}{2}}\left(\partial\Omega\right)}}d\rho=\int_{0}^{\infty}\rho^{-\theta}\left\langle u,\left(\left(\rho+A\right)^{-1}\right)^{*}v\right\rangle _{H^{-\frac{1}{2}}\left(\partial\Omega\right)\times H^{\frac{1}{2}}\left(\partial\Omega\right)}d\rho
\nonumber\\
&=&
\int_{0}^{\infty}\rho^{-\theta}\left\langle u,\left(\rho+A^{*}\right)^{-1}v\right\rangle _{H^{-\frac{1}{2}}\left(\partial\Omega\right)\times H^{\frac{1}{2}}\left(\partial\Omega\right)}d\rho=\frac{\pi}{\sin\left(\theta\pi\right)}\left\langle u,\left(A^{*}\right)^{-\theta}v\right\rangle _{H^{-\frac{1}{2}}\left(\partial\Omega\right)\times H^{\frac{1}{2}}\left(\partial\Omega\right)}.
\nonumber
\end{eqnarray}
\end{proof}
As $A\left(t\right)$, $\left.A\left(t\right)\right|_{L^{2}\left(\partial\Omega\right)}$
and $\left.A\left(t\right)\right|_{H^{\frac{1}{2}}\left(\partial\Omega\right)}$
are sectorial, we can define $A\left(t\right)^{-\theta}\in\mathcal{B}\left(H^{-\frac{1}{2}}\left(\partial\Omega\right)\right)$,
$\left(\left.A\left(t\right)\right|_{L^{2}\left(\partial\Omega\right)}\right)^{-\theta}\in\mathcal{B}\left(L^{2}\left(\partial\Omega\right)\right)$
and $\left(\left.A\left(t\right)\right|_{H^{\frac{1}{2}}\left(\partial\Omega\right)}\right)^{-\theta}\in\mathcal{B}\left(H^{\frac{1}{2}}\left(\partial\Omega\right)\right)$
for all $\theta\in\left]0,1\right[$. By the definition of the fractional
powers of an operator, it is clear that $A\left(t\right)^{-\theta}$
coincides with $\left(\left.A\left(t\right)\right|_{L^{2}\left(\partial\Omega\right)}\right)^{-\theta}$
and $\left(\left.A\left(t\right)\right|_{H^{\frac{1}{2}}\left(\partial\Omega\right)}\right)^{-\theta}$
in $L^{2}\left(\partial\Omega\right)$ and $H^{\frac{1}{2}}\left(\partial\Omega\right)$,
respectively. In particular, $A\left(t\right)^{-\theta}$ takes elements
from $L^{2}\left(\partial\Omega\right)$ into $L^{2}\left(\partial\Omega\right)$
and elements from $H^{\frac{1}{2}}\left(\partial\Omega\right)$ into
$H^{\frac{1}{2}}\left(\partial\Omega\right)$. Actually we can say
a little more about their mapping properties.
\begin{lem}\label{lem:lemma20}
Let $\theta\in\left]\frac{1}{2},1\right]$. If $y\in L^{2}\left(\partial\Omega\right)$,
then $A\left(t\right)^{-\theta}y\in H^{\frac{1}{2}}\left(\partial\Omega\right)$
and $A\left(t\right)^{-\theta}:L^{2}\left(\partial\Omega\right)\to H^{\frac{1}{2}}\left(\partial\Omega\right)$
is a continuous operator. In particular, $\left(A\left(t\right)^{*}\right)^{-\theta}:L^{2}\left(\partial\Omega\right)\to H^{\frac{1}{2}}\left(\partial\Omega\right)$ is continuous.
\end{lem}
\begin{proof}
First, note that $\left(H^{-\frac{1}{2}}\left(\partial\Omega\right),\mathcal{D}\left(A\left(t\right)\right)\right)_{\frac{1}{2},\infty}\subset\mathcal{D}\left(A\left(t\right)^{\sigma}\right)$, \cite[Proposition 1.1.4 and 4.1.7]{LunardiInterpolation}, for all $0<\sigma<\frac{1}{2}$, where the real interpolation space $\left(H^{-\frac{1}{2}}\left(\partial\Omega\right),\mathcal{D}\left(A\left(t\right)\right)\right)_{\frac{1}{2},\infty}$ is equal to 
\[\left\{ y\in H^{-\frac{1}{2}}\left(\partial\Omega\right);\,\text{sup}_{\rho\in\left[0,\infty\right[}\left\Vert \rho^{\frac{1}{2}}A\left(t\right)\left(\rho+A\left(t\right)\right)^{-1}y\right\Vert_{H^{-\frac{1}{2}}(\partial\Omega)} <\infty\right\}
\]
by \cite[Propositions 2.2.2 and 2.2.6]{Lunardi}. We divide the proof into steps.

First step: If $\left(A\left(t\right)-\lambda\right)y=f$, for $y\in H^{\frac{1}{2}}\left(\partial\Omega\right)$,
$f\in L^{2}\left(\partial\Omega\right)$ and $Re\left(\lambda\right)<0$,
then 
$$\left\Vert y\right\Vert _{H^{\frac{1}{2}}\left(\partial\Omega\right)}\le C\left|\lambda\right|^{-\frac{1}{2}}\left\Vert f\right\Vert _{L^{2}\left(\partial\Omega\right)}.$$

In order to prove it, let $u_{t}=\mathcal{E}(y)-\mathcal{B}_{t,D}^{-1}\left(P(t,x,D)\mathcal{E}(y)\right)$. Since $y=\gamma_{0}(u_{t})$, we have 
\[
\left\Vert y\right\Vert _{H^{\frac{1}{2}}\left(\partial\Omega\right)}^{2} \overset{(1)}{\le} C\left\Vert u_{t}\right\Vert _{H^{1}\left(\Omega\right)}^{2} \overset{(2)}{\le} C\,\text{Re}\left(u_{t},u_{t}\right) \overset{(3)}{\le} C\left\Vert f\right\Vert _{L^{2}\left(\partial\Omega\right)}\left\Vert \gamma_{0}\left(u_{t}\right)\right\Vert _{L^{2}\left(\partial\Omega\right)} \overset{(4)}{\le} C\left|\lambda\right|^{-1}\left\Vert f\right\Vert _{L^{2}\left(\partial\Omega\right)}^{2}.
\]
We have used the continuity of the trace  in $(1)$, Equation \eqref{eq:Rea(u,u)geu} in $(2)$, Equation \eqref{eq:aaa} in $(3)$ and Proposition \ref{prop:L2sectorial} in $(4)$. The constants $C>0$ can change from one inequality to another.

Second step: If $0<\sigma<\frac{1}{2}$, then $L^{2}\left(\partial\Omega\right)\subset\mathcal{D}\left(A\left(t\right)^{\sigma}\right)$
and the inclusion is continuous.

Let $y\in L^{2}\left(\partial\Omega\right)$. Then 
\begin{eqnarray}
\rho^{\frac{1}{2}}\left\Vert A\left(t\right)\left(\rho+A\left(t\right)\right)^{-1}y\right\Vert _{H^{-\frac{1}{2}}\left(\partial\Omega\right)}&\le&\left\Vert A\left(t\right)\right\Vert _{\mathcal{B}\left(H^{\frac{1}{2}}\left(\partial\Omega\right),H^{-\frac{1}{2}}\left(\partial\Omega\right)\right)}\rho^{\frac{1}{2}}\left\Vert \left(\rho+A\left(t\right)\right)^{-1}y\right\Vert _{H^{\frac{1}{2}}\left(\partial\Omega\right)}
\nonumber\\
&\le&C \left\Vert A\left(t\right)\right\Vert _{\mathcal{B}\left(H^{\frac{1}{2}}\left(\partial\Omega\right),H^{-\frac{1}{2}}\left(\partial\Omega\right)\right)}\left\Vert y\right\Vert _{L^{2}\left(\partial\Omega\right)}.
\nonumber
\end{eqnarray}
This implies that $y\in\left(H^{-\frac{1}{2}}\left(\partial\Omega\right),\mathcal{D}\left(A\left(t\right)\right)\right)_{\frac{1}{2},\infty}\subset\mathcal{D}\left(A\left(t\right)^{\sigma}\right)$
and that the inclusions are continuous. Hence $\left\Vert A\left(t\right)^{\sigma}y\right\Vert _{H^{-\frac{1}{2}}\left(\partial\Omega\right)}\le C\left\Vert y\right\Vert _{L^{2}\left(\partial\Omega\right)}$, where the constant $C$ does not depend on $t$, due to the uniform boundedness of $A(t)$ and its resolvent $\left(\rho+A(t)\right)^{-1}$, for $\rho>0$.

Third step: If $\theta\in\left]\frac{1}{2},1\right]$, then $A\left(t\right)^{-\theta}:L^{2}\left(\partial\Omega\right)\to H^{\frac{1}{2}}\left(\partial\Omega\right)$
continuously.

Let $y\in H^{-\frac{1}{2}}\left(\partial\Omega\right)$ and $x\in H^{\frac{1}{2}}\left(\partial\Omega\right)$.
Then
\begin{eqnarray}
& &\left|\left\langle y,A\left(t\right)^{-\theta}x\right\rangle _{H^{-\frac{1}{2}}\left(\partial\Omega\right)\times H^{\frac{1}{2}}\left(\partial\Omega\right)}\right| = \left|\left\langle y,A\left(t\right)^{-1}A\left(t\right)^{1-\theta}x\right\rangle _{H^{-\frac{1}{2}}\left(\partial\Omega\right)\times H^{\frac{1}{2}}\left(\partial\Omega\right)}\right|
\nonumber\\
&=& \left|\overline{\left\langle A\left(t\right)^{1-\theta}x,\left(A\left(t\right)^{-1}\right)^{*}y\right\rangle _{H^{-\frac{1}{2}}\left(\partial\Omega\right)\times H^{\frac{1}{2}}\left(\partial\Omega\right)}}\right| \le \left\Vert A\left(t\right)^{1-\theta}x\right\Vert _{H^{-\frac{1}{2}}\left(\partial\Omega\right)}\left\Vert \left(A\left(t\right)^{-1}\right)^{*}y\right\Vert _{H^{\frac{1}{2}}\left(\partial\Omega\right)}
\nonumber\\
& \le & C \left\Vert y\right\Vert _{H^{-\frac{1}{2}}\left(\partial\Omega\right)}\left\Vert x\right\Vert _{L^{2}\left(\partial\Omega\right)}.
\nonumber
\end{eqnarray}
Hence $\left\Vert A\left(t\right)^{-\theta}x\right\Vert _{H^{\frac{1}{2}}\left(\partial\Omega\right)}\le C\left\Vert x\right\Vert _{L^{2}\left(\partial\Omega\right)}$, where $C$ is again a constant that does not depend on $t$.
As $H^{\frac{1}{2}}\left(\partial\Omega\right)$ is dense in $L^{2}\left(\partial\Omega\right)$,
we obtain the result.

Finally, we see that $\left(A\left(t\right)^{*}\right)^{-\theta}:L^{2}\left(\partial\Omega\right)\to H^{\frac{1}{2}}\left(\partial\Omega\right)$
is also a continuous operator, as $A\left(t\right)^{*}$ is the operator
associated to the sesquilinear form $a_{t}^{*}:H^{1}\left(\Omega\right)\times H^{1}\left(\Omega\right)\to\mathbb{C}$
defined as $a_{t}^{*}\left(u,v\right)=\overline{a_{t}\left(v,u\right)}$,
which has the same properties of $a_{t}$.
\end{proof}
We now proceed to prove Theorem \ref{thm:Yagi}.
\begin{proof}
(of the Theorem \ref{thm:Yagi})

We have to check all conditions of Definition \ref{def:Yagiconditions}. The Item $1)$ follows from Proposition \ref{prop:L2sectorial}. It remains to prove Item $2)$.

For all $x,y\in H^{\frac{1}{2}}\left(\partial\Omega\right)$, we have 	

\begin{eqnarray}
& & \left|\left(\left.A\left(t\right)\right|_{L^{2}\left(\partial\Omega\right)}^{\nu}\left(\left.A\left(t\right)\right|_{L^{2}\left(\partial\Omega\right)}^{-1}-\left.A\left(s\right)\right|_{L^{2}\left(\partial\Omega\right)}^{-1}\right)x,y\right)_{L^{2}\left(\partial\Omega\right)}\right|
\nonumber\\	
&=&	\left|\left\langle A\left(t\right)^{\nu}\left(A\left(t\right)^{-1}-A\left(s\right)^{-1}\right)x,y\right\rangle _{H^{-\frac{1}{2}}\left(\partial\Omega\right)\times H^{\frac{1}{2}}\left(\partial\Omega\right)}\right|	
\nonumber\\
&=&	\left|\left\langle A\left(t\right)^{\nu-1}A\left(t\right)\left(A\left(t\right)^{-1}-A\left(s\right)^{-1}\right)x,y\right\rangle _{H^{-\frac{1}{2}}\left(\partial\Omega\right)\times H^{\frac{1}{2}}\left(\partial\Omega\right)}\right|
\nonumber\\	
&\overset{(1)}{=}&	\left|\left\langle A\left(t\right)\left(A\left(t\right)^{-1}-A\left(s\right)^{-1}\right)x,\left(A\left(t\right)^{*}\right)^{\nu-1}y\right\rangle _{H^{-\frac{1}{2}}\left(\partial\Omega\right)\times H^{\frac{1}{2}}\left(\partial\Omega\right)}\right|
\nonumber\\
&\overset{(2)}{\le}& C\left\Vert \left(A\left(s\right)^{-1}-A\left(t\right)^{-1}\right)x\right\Vert _{H^{\frac{1}{2}}\left(\partial\Omega\right)}\left\Vert \left(A\left(t\right)^{*}\right)^{\nu-1}y\right\Vert _{H^{\frac{1}{2}}\left(\partial\Omega\right)}
\nonumber\\
&\overset{(3)}{\le}& C\left|t-s\right|^{\alpha}\left\Vert x\right\Vert _{H^{-\frac{1}{2}}\left(\partial\Omega\right)}\left\Vert \left(A\left(t\right)^{*}\right)^{\nu-1}y\right\Vert _{H^{\frac{1}{2}}\left(\partial\Omega\right)}\overset{(4)}{\le}C\left|t-s\right|^{\alpha}\left\Vert x\right\Vert _{L^{2}\left(\partial\Omega\right)}\left\Vert y\right\Vert _{L^{2}\left(\partial\Omega\right)}.
\nonumber
\end{eqnarray}

We have used Lemma \ref{lem:lemma19} in $(1)$, Proposition \ref{cor:uniforAt} in $(2)$. In $(3)$, we use that $A\left(t\right)^{-1}=\gamma_{0}\circ\mathcal{B}_{t,N}^{-1}\circ k$ as proved in Proposition \ref{cor:uniforAt} and that $\mathcal{B}_{t,N}^{-1}$ is Hölder continuous as proved in Theorem \ref{thm:teorema11}. Finally, in $(4)$, we have used Lemma \ref{lem:lemma20}.

\end{proof}

\section{Applications}

\subsection{Non-autonomous elliptic equations with dynamic boundary conditions.}

In this section, we first consider the following problem:
\begin{equation}
\left\{
\begin{array}{rcl}
P\left(t,x,D\right)u\left(t,x\right) & = & 0,\,(t,x)\in\left]t_{0},\infty\right[\times\Omega\\
\frac{\partial u}{\partial t}\left(t,x\right) & = & -C\left(t,x,D\right)u\left(t,x\right)+f\left(t,x\right),\,(t,x)\in\left]t_{0},\infty\right[\times\partial\Omega\\
u\left(t_{0},x\right) & = & u_{0}\left(x\right),\,x\in\partial\Omega,
\end{array}\label{p890}
\right.
\end{equation}
where $\Omega\subset\mathbb{R}^{n}$ is a bounded Lipschitz domain,
$P\left(t,x,D\right)$ and $C\left(t,x,D\right)$ are the operators
defined in (\ref{eq:P(t,x,D)}) and (\ref{eq:dirichlettoneumann}) with coefficients that satisfy the conditions
of Assumption \ref{fact:Assumptionsoftheform} for the $H^{-\frac{1}{2}}(\Omega)$ problem and Assumption 2 for the $L^2(\Omega)$ problem.
\begin{thm}
\label{thm:maintheoremnonautonomouselliptic} Suppose that the Assumption 1 holds, that is, $\alpha\in]0,1]$, and let
$u_{0}\in H^{-\frac{1}{2}}\left(\partial\Omega\right)$ and $f\in C_{u}^{\alpha}\left(\left[t_{0},\infty\right[,H^{-\frac{1}{2}}\left(\partial\Omega\right)\right)$.
Then, there is a unique function $u\in C\left(\left[t_{0},\infty\right[,H^{1}\left(\Omega\right)\right)$
such that:

1) $\gamma_{0}\left(u\right)\in C^{1}\left(\left]t_{0},\infty\right[,H^{-\frac{1}{2}}\left(\partial\Omega\right)\right)\cap C\left(\left[t_{0},\infty\right[,H^{-\frac{1}{2}}\left(\partial\Omega\right)\right)\cap C\left(\left]t_{0},\infty\right[,H^{\frac{1}{2}}\left(\partial\Omega\right)\right)$.

2) 
\begin{equation}\label{solution}
\left\{
\begin{array}{rcl}
P\left(t,x,D\right)u\left(t\right) &=& 0,\forall t\in\left[t_{0},\infty\right[ \\
\frac{d\gamma_{0}(u)}{dt}\left(t\right) &=& -C\left(t,x,D\right)u\left(t\right)+f\left(t\right),\,t\in\left]t_{0},\infty\right[\\
\gamma_{0}(u)\left(t_{0}\right) &=& u_{0}
\end{array}.
\right.
\end{equation}

Moreover, if $\lim_{t\to\infty}f\left(t\right)=f_{\infty}$ in $H^{-\frac{1}{2}}\left(\partial\Omega\right)$, then $\lim_{t\to\infty}u\left(t\right)=u_{\infty}$ in $H^{1}\left(\Omega\right)$, where $u_{\infty}$ is the unique solution of 
\begin{equation}\label{eq:infty}
\left\{
\begin{array}{rcl}
P\left(\infty,x,D\right)u_{\infty} & = & 0\\
C\left(\infty,x,D\right)u_{\infty} & = & f_{\infty}
\end{array}.
\right.
\end{equation}

If the Assumption 2 holds, that is, $\alpha\in]\frac{1}{2},1]$, and if $u_{0}\in L^{2}\left(\partial\Omega\right)$ and $f\in C_{u}^{\alpha}\left(\left[t_{0},\infty\right[,L^{2}\left(\partial\Omega\right)\right)$, then there is a unique $u\in C\left(\left[t_{0},\infty\right[,H^{1}\left(\Omega\right)\right)$ such that its trace $\gamma_{0}\left(u\right)\in C^{1}\left(\left]t_{0},\infty\right[,L^{2}\left(\partial\Omega\right)\right)\cap C\left(\left[t_{0},\infty\right[,L^{2}\left(\partial\Omega\right)\right)\cap C\left(\left]t_{0},\infty\right[,H^{\frac{1}{2}}\left(\partial\Omega\right)\right)$ and \eqref{solution} holds. In particular, $C\left(t,x,D\right)u\left(t\right)$ exists in the $L^{2}(\partial\Omega)$-weak sense.

Moreover, if $\lim_{t\to\infty}f\left(t\right)=f_{\infty}$ in $L^{2}\left(\partial\Omega\right)$, then $\lim_{t\to\infty}u\left(t\right)=u_{\infty}$ in $H^{1}\left(\Omega\right)$, where $u_{\infty}$ is the unique solution of \eqref{eq:infty}. 
\end{thm}
\begin{proof}
Suppose that $u\in C\left(\left[t_{0},\infty\right[,H^{1}\left(\Omega\right)\right)$
satisfies the conditions of item 1 and 2 of the theorem. As $P\left(t,x,D\right)u\left(t\right)=0$,
the expression $C\left(t,x,D\right)u\left(t\right)$ is equivalent
to $A\left(t\right)\left(\gamma_{0}(u)(t)\right)$, where $A\left(t\right)$
is the Dirichlet-to-Neumann operator.
Hence $\gamma_{0}(u)$ must be the solution of the
Equation (\ref{eq:mainproblem}) and it is uniquely determined by $u_{0}$ and $f$. By Proposition \ref{lem:Dirichlet-Problem}, we conclude that 
\begin{equation}\label{eq:opa}
u\left(t\right)=\mathcal{E}\left(\gamma_{0}\left(u\right)\left(t\right)\right)-\mathcal{B}_{t,D}^{-1}\left(P\left(t,x,D\right)\mathcal{E}\left(\gamma_{0}\left(u\right)\left(t\right)\right)\right).
\end{equation}
Thus $u$ is also uniquely determined by $u_{0}$ and $f$. On the other hand, if we use \eqref{eq:opa} as the definition of $u\left(t\right)$, where $\gamma_{0}\left(u\right)$
is the solution of Equation (\ref{eq:mainproblem}),
then $u$ satisfies properties 1 and 2 stated in the theorem. This proves existence.

If $\lim_{t\to\infty}f\left(t\right)=f_{\infty}$ in $H^{-\frac{1}{2}}\left(\partial\Omega\right)$,
then, by Corollary \ref{cor:dthhmeio}, $\lim_{t\to\infty}\gamma_{0}\left(u\right)\left(t\right)=\gamma_{0}\left(u\right)\left(\infty\right)$
in $H^{\frac{1}{2}}\left(\partial\Omega\right)$, where $A\left(\infty\right)\left(\gamma_{0}\left(u\right)\left(\infty\right)\right)=f_{\infty}$.
Hence $\lim_{t\to\infty}u\left(t\right)=u_{\infty}$ in $H^{1}\left(\Omega\right)$,
where \[u_{\infty}=\mathcal{E}\left(\gamma_{0}\left(u\right)\left(\infty\right)\right)-\mathcal{B}_{t,D}^{-1}\left(P\left(t,x,D\right)\mathcal{E}\left(\gamma_{0}\left(u\right)\left(\infty\right)\right)\right)\] and, therefore, it is the unique solution of Equation \eqref{eq:infty}. The results for $L^{2}\left(\partial\Omega\right)$ case follow from similar arguments and Theorem \ref{thm:L2}.
\end{proof}

\subsection{Dynamical
boundary conditions on non-cylindrical domains.}\label{sec:3.2}.

Finally, we consider the Laplace equation with dynamic
boundary conditions on a non-cylindrical domain as described in the Introduction:
\begin{equation}
\left\{
\begin{gathered}
\left(\lambda+\Delta\right)u\left(t,x\right) = 0,\,(t,x)\in D, t> t_{0}\\
\frac{\partial u}{\partial t}\left(t,x\right) = -\frac{\partial u}{\partial n}\left(t,x\right)+f\left(t,x\right),\,(t,x)\in S, t> t_{0} \\
u\left(t_{0},x\right) = u_{0}\left(x\right),\,x\in\partial\Omega_{t_{0}},
\label{eq:MainProblem2}
\end{gathered}
\right.
\end{equation}
where $\lambda<0$ and $t_{0}\ge0$.


In order to define the set $D$, we consider a bounded domain $\Omega\subset\mathbb{R}^{n}$ with
$C^{2}$-regular boundary $\partial\Omega$ and a map
\begin{equation}\label{eq:proph}
h \in C_{u}^{\alpha}\left(\left[0,\infty\right[,\text{Diff}^{2}\left(\Omega\right)\right)\cap C^{1,\alpha}_{u}\left(\left[0,\infty\right[, C \left(\overline{\Omega},\mathbb{R}^{n}\right)\right)
\end{equation}
where $\text{Diff}^{2}\left(\Omega\right)$ is an open set of $C^{2}\left(\overline{\Omega},\mathbb{R}^{n}\right)$, defined as 
\[
\text{Diff}^{2}\left(\Omega\right):=\left\{ g\in C^{2}\left(\overline{\Omega},\mathbb{R}^{n}\right);\,g\,\text{is injective and }\det\left(\frac{\partial g_{i}}{\partial x_{j}}\left(y\right)\right)\ne0,\,\forall y\in\overline{\Omega}\right\} .
\]
The set $D\subset\mathbb{R}^{n+1}$ is defined as the image of the function
$\mathcal{H}:\left[0,\infty\right[\times\Omega\to\left[0,\infty\right[\times\mathbb{R}^{n}$,
 given by $\mathcal{H}\left(t,y\right)=\left(t,h\left(t,y\right)\right).$ According to (\ref{eq:proph}), $D$ is an open set of $\mathbb{R}^{n+1}$ since $\mathcal{H}$ is 
a diffeomorphism onto its image. 
We notice that $h$ defines a family of diffeomorphisms $\left\{h_t:\Omega\to\Omega_{t}\right\}_{t\ge 0}$ given as $h_{t}(y)=h(t,y)$ and that
$$
\Omega_{t}=\left\{ \left(t,h_{t}\left(y\right)\right);\,y\in\Omega\right\}, \quad
\partial\Omega_{t}=\left\{ \left(t,h_{t}\left(y\right)\right);\,y\in\partial\Omega\right\}, \; \textrm{ for } t \ge 0, 
$$
and $S=h\left(\left]0,\infty\right[\times\partial\Omega\right)$.

\begin{fact}
\label{fact:Assumptionsofh}The function $h$ satisfies: 

i) $\lim_{t\to\infty}h\left(t,.\right)=I$
in $C^{2}\left(\overline{\Omega},\mathbb{R}^{n}\right)$ where $I:\Omega\to\Omega$
is the identity.


ii) There is a function $c\in C_{u}^{\alpha}\left(\left[0,\infty\right[\right)$, where $\alpha\in\left]\frac{1}{2},1\right]$ such that $\frac{\partial h}{\partial t}\left(t,y\right)=c\left(t\right)n\left(t,h\left(t,y\right)\right)$,
for all $y\in\partial\Omega$, where $n\left(t,h\left(t,y\right)\right)$
is the outward normal vector to $\Omega_{t}$ at the point $h\left(t,y\right)\in\partial\Omega_{t}$. 
\end{fact}  
The item i) gives a precise meaning to the convergence of the sets
$\Omega_{t}$ to $\Omega$ as $t\to\infty$. Intuitively it also says
that $\Omega_{t}$ are temporal perturbations of the set $\Omega$.
The Hölder continuity of $h$ and $\frac{dh}{dt}$ assumed in (\ref{eq:proph}) implies that
the perturbations and its rate of variation do not change rapidly.
The second item of Assumption \ref{fact:Assumptionsofh} says that we allow only small perturbations
of the domain along the normal vector.

\begin{example}
Let $\Omega$ be a bounded domain with $C^{3}$-regular boundary. Let $\nu:\partial\Omega\to\mathbb{R}^{n}$
be the outward normal vector field and $B_{r}^{N}\left(y\right):=\left\{ y+t\nu\left(y\right),\,\left|t\right|\le r\right\} $, for $y\in\partial\Omega$. By the collar neighborhood theorem, there is an $r>0$ such
that $U:=\cup_{y\in\partial\Omega}B_{r}^{N}\left(y\right)$ is an
open set and $B_{r}^{N}\left(w\right)\cap B_{r}^{N}\left(y\right)=\emptyset$,
if $w\ne y$, $w,y\in\partial\Omega$. Moreover there is a unique
function $\pi:U\to\partial\Omega$ of class $C^{2}$ such that $\pi\left(z\right)=y$,
when $z\in B_{r}^{N}\left(y\right)$. 
Let $\chi\in C_{c}^{\infty}\left(U\right)$ be equal to 1 in a neighborhood
of $\partial\Omega$. Now take $f\in C^{2}\left(\left[0,\infty\right[\right)$
and define $h:\left[0,\infty\right[\times\Omega\to\mathbb{R}^{n}$
as 
\[
h\left(t,y\right)=y+f\left(t\right)\chi\left(y\right)\nu\left(\pi\left(y\right)\right).
\]
If $\left\Vert f\right\Vert _{L^{\infty}\left(\left[0,\infty\right[\right)}$
is small, it is clear that the above formula defines a $C^{2}$ diffeomorphism,
for all $t\in\left[0,\infty\right[$. Moreover, take $y\in\partial\Omega$
and let $\phi:B_{1}\left(0\right)=\left\{w\in\mathbb{R}^{n-1};\|w\|<1\right\}\subset\mathbb{R}^{n-1}\to\partial\Omega$
be an embedding of class $C^{3}$ such that $\phi\left(0\right)=y$.
Hence $x\in B_{1}\left(0\right)\mapsto h\left(t,\phi\left(x\right)\right)\in\partial\Omega_{t}$
is an embedding of class $C^{2}$. The tangent space $T_{h\left(t,y\right)}\partial\Omega_{t}$
consists of the linear span of the vectors $\left.\frac{\partial}{\partial x_{j}}h\left(t,\phi\left(x\right)\right)\right|_{x=0}$,
$j=1,...,n-1$. Denoting by $\left\langle a,b\right\rangle_{\mathbb{R}^{n}}$ the usual scalar product of vectors $a$ and $b$ in $\mathbb{R}^{n}$, we have

\begin{eqnarray}
\left\langle \frac{\partial}{\partial x_{j}}\left(h\left(t,\phi\left(x\right)\right)\right),\nu\left(\phi\left(x\right)\right)\right\rangle_{\mathbb{R}^{n}} &=&\left\langle \frac{\partial\phi}{\partial x_{j}}\left(x\right),\nu\left(\phi\left(x\right)\right)\right\rangle_{\mathbb{R}^{n}} +f\left(t\right)\left\langle \frac{\partial}{\partial x_{j}}\nu\left(\phi\left(x\right)\right),\nu\left(\phi\left(x\right)\right)\right\rangle_{\mathbb{R}^{n}}
\nonumber\\
&=&
\left\langle \frac{\partial\phi}{\partial x_{j}}\left(x\right),\nu\left(\phi\left(x\right)\right)\right\rangle_{\mathbb{R}^{n}} +f\left(t\right)\frac{1}{2}\frac{\partial}{\partial x_{j}}\left\langle \nu\left(\phi\left(x\right)\right),\nu\left(\phi\left(x\right)\right)\right\rangle_{\mathbb{R}^{n}} \nonumber\\
&=&0.
\nonumber
\end{eqnarray}
Hence $\nu\left(y\right)=\nu\left(\phi\left(0\right)\right)$
is the normal vector at $h\left(t,y\right)$, that is, $n\left(t,h\left(t,y\right)\right)=\nu\left(y\right)$.
In particular, $\frac{\partial h}{\partial t}\left(t,y\right)=\frac{df}{dt}\left(t\right)n\left(t,h\left(t,y\right)\right)$, when $y\in\partial\Omega$.
Hence, all items of {Assumption 3} hold if $\|f\|_{L^{\infty}}$ is sufficiently
small,  if the first and second derivatives of $f$ are bounded and if $\lim_{t\to\infty}f\left(t\right)=0$.
In particular, we can take $f$ behaving as $t^{-\beta} \sin t^\alpha$, for $t$ large enough, and positive constants $\alpha$ and $\beta$ satisfying $2(\alpha-1)<\beta$.
Consequently, our assumptions also allow a kind of oscillatory behavior to the boundary $S$ of the non-cylindrical domain $D$ at infinite time. 
\end{example}

\begin{rem}\label{rem:15}
{Assumption 3} implies the following convergences, for all $i,j,k\in\left\{ 1,...,n\right\}$:
$$
\lim_{t\to\infty}\left\Vert \frac{\partial h_{i}}{\partial y_{j}}\left(t,.\right)-\delta_{ij}\right\Vert _{L^{\infty}\left(\Omega\right)}=\lim_{t\to\infty}\left\Vert \frac{\partial h_{i}}{\partial y_{j}\partial y_{k}}\left(t,.\right)\right\Vert _{L^{\infty}\left(\Omega\right)}=\lim_{t\to\infty}\left\Vert \frac{\partial h}{\partial t}\left(t,.\right)\right\Vert _{L^{\infty}\left(\Omega\right)^{n}}=0.
$$
The first two limits follow directly from to item $i)$ of Assumption 3.
For the third one, we consider $\left(t,y\right)$ such that $\frac{\partial h_{i}}{\partial t}\left(t,y\right)>0$,
for some $i\in\left\{ 1,...,n\right\} $, and $k:=\left(\frac{1}{2C}\frac{\partial h_{i}}{\partial t}\left(t,y\right)\right)^{\frac{1}{\alpha}}$,
where $C>0$ is the constant of Hölder continuity expressed in (\ref{eq:proph}). Then
\[
\left\Vert \frac{\partial h_{i}}{\partial t}\left(t+\theta k,.\right)-\frac{\partial h_{i}}{\partial t}\left(t,.\right)\right\Vert _{L^{\infty}\left(\Omega\right)}\le C\left(\theta k\right)^{\alpha}\le\frac{1}{2}\frac{\partial h_{i}}{\partial t}\left(t,y\right).
\]
Therefore, for all $\theta\in\left[0,1\right]$, we have $\frac{\partial h_{i}}{\partial t}\left(t+\theta k,y\right)\ge\frac{1}{2}\frac{\partial h_{i}}{\partial t}\left(t,y\right)$.
We then conclude that
\[
\frac{1}{2}\left(\frac{1}{2C}\right)^{\frac{1}{\alpha}}\left|\frac{\partial h_{i}}{\partial t}\left(t,y\right)\right|^{\frac{1+\alpha}{\alpha}}\le k\left|\int_{0}^{1}\frac{\partial h_{i}}{\partial t}\left(t+\theta k,y\right)d\theta\right|\le\left|h_{i}\left(t+k,y\right)-h_{i}\left(t,y\right)\right|.
\]
The above estimate holds also if $\frac{\partial h_{i}}{\partial t}\left(t,y\right)<0$
by same arguments. As $\lim_{t\to\infty}\left\Vert h\left(t,.\right)-I\right\Vert _{L^{\infty}\left(\Omega\right)^{n}}=0$,
we conclude that $\lim_{t\to\infty}\left\Vert h_{i}\left(t+k,.\right)-h_{i}\left(t,.\right)\right\Vert _{L^{\infty}\left(\Omega\right)}=0$ and that $\left\Vert \frac{\partial h}{\partial t}\left(t,.\right)\right\Vert _{L^{\infty}\left(\Omega\right)^{n}}$
converges to zero. In particular, $\lim_{t\to\infty} c\left(t\right)=0$.
\end{rem}
In order to understand the Problem (\ref{eq:MainProblem2}), let us consider a function $u:\left\{ \left(t,x\right)\in D;\,t>t_{0}\right\}\to\mathbb{C}$ that can be extended to a continuous
function in $\left\{ \left(t,x\right)\in\overline{D};\,t\ge t_{0}\right\}$ and such that $\frac{\partial u}{\partial t}$,
$\frac{\partial u}{\partial x_{j}}$ and $\frac{\partial^{2}u}{\partial x_{i}\partial x_{j}}$
exist and are continuous up to $\left\{ \left(t,x\right)\in\overline{D};\,t>t_{0}\right\} $, for all $i, j$.
Suppose that $u$ is a classical solution of the Problem (\ref{eq:MainProblem2}).
It is worth to mention that, in a point $\left(t,x\right) \in \left]t_{0},\infty\right[\times S$, $\frac{\partial u}{\partial t}\left(t,x\right)$ must be interpreted
as the continuous extension of $\frac{\partial u}{\partial t}$ to
this point. In fact, the limit $\lim_{h\to0}\frac{u\left(t+h,x\right)-u\left(t,x\right)}{h}$
does not always make sense, as it is not even clear that $\left(t+h,x\right)\in \overline{D}$
for some $h\ne0$, when $\left(t,x\right)\in S$.

For such a function, we can make a change of variables as in \cite[Chapter 2]{Henryboundary}. We define $v\left(t,y\right):=u\left(t,h\left(t,y\right)\right)$, and consider the matrix $\left(\frac{\partial h}{\partial y}\right)$, whose entries are $\left(\frac{\partial h}{\partial y}\right)_{kj}\left(t,y\right):=\frac{\partial h_{k}}{\partial y_{j}}\left(t,y\right)$, and $\left(\frac{\partial h}{\partial y}\right)^{-1}$ is the inverse of this matrix. Moreover we denote by $\nu\left(y\right)$ the normal vector at $y\in\partial\Omega$ and by $n\left(t,x\right)$ the normal vector at $x\in \partial \Omega_{t}$. We then have

\begin{equation}
\begin{array}{lcl}
i) \frac{\partial u}{\partial x_{j}}\left(t,h\left(t,y\right)\right)&=&\sum_{k=1}^{n}\left(\frac{\partial h}{\partial y}\right)_{kj}^{-1}\left(t,y\right)\frac{\partial v}{\partial y_{k}}\left(t,y\right).
\\
ii)
\frac{\partial u}{\partial t}\left(t,h\left(t,y\right)\right)&=&\frac{\partial v}{\partial t}\left(t,y\right)-\sum_{l=1}^{n}\left(\sum_{k=1}^{n}\left(\frac{\partial h}{\partial y}\right)_{lk}^{-1}\left(t,y\right)\frac{\partial h_{k}}{\partial t}\left(t,y\right)\right)\frac{\partial v}{\partial y_{l}}\left(t,y\right).
\\
iii) 
\frac{\partial u}{\partial n}\left(t,h\left(t,y\right)\right)&=&\sum_{k=1}^{n}\left(\sum_{j=1}^{n}\left(\frac{\partial h}{\partial y}\right)_{kj}^{-1}\left(t,y\right)n_{j}\left(t,h\left(t,y\right)\right)\right)\frac{\partial v}{\partial y_{k}}\left(t,y\right).
\\
iv)\,\,\,\,\,\,\,\
\nu_{j}\left(y\right)&=&\frac{\sum_{k=1}^{n}\left(\frac{\partial h}{\partial y}\right)_{kj}\left(t,y\right)n_{k}\left(t,h\left(t,y\right)\right)}{\sqrt{\sum_{j=1}^{n}\left(\sum_{k=1}^{n}\left(\frac{\partial h}{\partial y}\right)_{kj}\left(t,y\right)n_{k}\left(t,h\left(t,y\right)\right)\right)^{2}}}.\label{eq:nunormal}
\end{array}
\end{equation}

Using item ii) of Assumption \ref{fact:Assumptionsofh},
we conclude that the Equation (\ref{eq:MainProblem2})
is formally equivalent to the following non-autonomous
elliptic equation with dynamic boundary conditions: 
\begin{equation} \label{eq:equivalentequation}
\begin{array}{rcl}
0 & = & \lambda v+\sum_{j=1}^{n}\sum_{k=1}^{n}\sum_{l=1}^{n}\left(\frac{\partial h}{\partial y}\right)_{lj}^{-1}\left(t,y\right)\frac{\partial}{\partial y_{l}}\left(\left(\frac{\partial h}{\partial y}\right)_{kj}^{-1}\left(t,y\right)\frac{\partial v}{\partial y_{k}}\left(t,y\right)\right),\,\left]t_{0},\infty\right[\times\Omega
\\
\frac{\partial v}{\partial t}\left(t,y\right) & = & \sum_{l=1}^{n}\left(\sum_{k=1}^{n}\left(\frac{\partial h}{\partial y}\right)_{lk}^{-1}\left(t,y\right)\left(c\left(t\right)-1\right)n_{k}\left(t,h\left(t,y\right)\right)\right)\frac{\partial v}{\partial y_{l}}\left(t,y\right)
\\
&&+f\left(t,h\left(t,y\right)\right),\,\left]t_{0},\infty\right[\times\partial\Omega
\\
v\left(t_{0 },y\right) & = & u_{0}\left(h\left(t_{0},y\right)\right),\,y\in\partial\Omega.
\end{array}
\end{equation}
The above equation can be studied using suitable forms. To define them, we fix a $C^{1}$ extension of $\nu:\partial\Omega\to\mathbb{R}^{n}$ to $\overline{\Omega}$ and call it, with a slight abuse of notation, by the same letter $\nu:\overline{\Omega}\to\mathbb{R}^{n}$. The normal vector $n$ can also be extended by the expression below:
\begin{equation}\label{eq:Normal}
n_{j}\left(t,h\left(t,y\right)\right)=\frac{\sum_{k=1}^{n}\left(\frac{\partial h}{\partial y}\right)_{kj}^{-1}\left(t,y\right)\nu_{k}\left(y\right)}{\sqrt{\sum_{j=1}^{n}\left(\sum_{k=1}^{n}\left(\frac{\partial h}{\partial y}\right)_{kj}^{-1}\left(t,y\right)\nu_{k}\left(y\right)\right)^{2}}}.
\end{equation}
Again we have used the same letter to denote its extension, $n=\left(n_{1},...,n_{n}\right):\overline{D}\to\mathbb{R}^{n}$.

Let $U\subset\mathbb{R}^{n}$ be an open set that contains $\partial\Omega$ and such that $\nu(y)\ne 0$ for all $y\in\overline{\Omega}\cap U$. We fix a function $\chi\in C_{c}^{\infty}\left(U\right)$ that satisfies $\left.\chi\right|_{\partial\Omega}\equiv 1$ and define $\mathcal{N}:\left[0,\infty\right[\times\overline{\Omega}\to\mathbb{R}$ as 
\begin{equation}\label{eq:defnormal}
\mathcal{N}\left(t,y\right)=\chi\left(y\right)\left(\sum_{j=1}^{n}\left(\sum_{k=1}^{n}\left(\frac{\partial h}{\partial y}\right)_{kj}^{-1}\left(t,y\right)\nu_{k}\left(y\right)\right)^{2}\right)^{-\frac{1}{2}}+1-\chi\left(y\right).
\end{equation}

This function has the following easily verified properties:

\begin{itemize}
\item $\mathcal{N}\left(t,y\right)>0$, for all $\left(t,y\right)\in\left[0,\infty\right[\times\overline{\Omega}$.
\item $t\in\left[0,\infty\right[\mapsto\left(x\in\Omega\mapsto\mathcal{N}\left(t,x\right)\right) \hbox{ is a function that belongs to } C_{u}^{\alpha}\left(\left[0,\infty\right[,C^{1}\left(\overline{\Omega}\right)\right)$.
\item $\lim_{t\to\infty}\left\Vert \mathcal{N}\left(t,.\right)-1\right\Vert _{C^{1}\left(\overline{\Omega}\right)}=0$.
\end{itemize} 

We finally set the forms $\left\{ a_{t}:H^{1}\left(\Omega\right)\times H^{1}\left(\Omega\right)\to\mathbb{C}\right\} _{t\in\left[t_{0},\infty\right]}$. The definition is obtained by the multiplication of Equation \eqref{eq:equivalentequation} by $\mathcal{N}(t,y)(c(t)-1)$ and integration by parts. We set
\begin{equation} \label{eq:formsfinal}
\begin{array}{lcl}
a_{t}\left(v,w\right) & = & \int_{\Omega}\left[\sum_{j=1}^{n}\sum_{k=1}^{n}\sum_{l=1}^{n}\left(\left(\frac{\partial h}{\partial y}\right)_{kj}^{-1}\left(t,y\right)\frac{\partial v}{\partial y_{k}}\left(y\right)\right)\right.
\\
&&\times
\left. \frac{\partial}{\partial y_{l}}\left(\left(\frac{\partial h}{\partial y}\right)_{lj}^{-1}\left(t,y\right)\left(1-c\left(t\right)\right)\mathcal{N}\left(t,y\right)\overline{w\left(y\right)}\right) - \lambda v\left(y\right)\left(1-c\left(t\right)\right)\mathcal{N}\left(t,y\right)\overline{w\left(y\right)}\right]dy 
\\
a_{\infty}\left(v,w\right) & = & \int_{\Omega}\left(\nabla v\left(y\right).\overline{\nabla w\left(y\right)}-\lambda v\left(y\right)\overline{w\left(y\right)}\right)dy.
\end{array}
\end{equation}

Items i) and ii) of Assumption \ref{fact:Assumptionsofh} and Remark \ref{rem:15} imply that the above forms satisfy the conditions of Assumption 2
for all $t\ge t_{0}$, if $t_{0} > 0$ is big enough. In particular, we can assume $c\left(t\right) < 1$ taking $t_{0}$ sufficiently large. 

The conormal derivative associated to the form $a_{\infty}$ of Equation (\ref{eq:formsfinal}) is the normal derivative $\frac{\partial }{\partial \nu}$. For $t_{0}<t<\infty$, we first note that Equations \eqref{eq:Normal} and \eqref{eq:defnormal} imply at $y\in\partial \Omega$, that
\[ \sqrt{\sum_{j=1}^{n}\left(\sum_{k=1}^{n}\left(\frac{\partial h}{\partial y}\right)_{kj}\left(t,y\right)n_{k}\left(t,h\left(t,y\right)\right)\right)^{2}}=\left(\sum_{j=1}^{n}\left(\sum_{k=1}^{n}\left(\frac{\partial h}{\partial y}\right)_{kj}^{-1}\left(t,y\right)\nu_{k}\left(y\right)\right)^{2}\right)^{-\frac{1}{2}}=\mathcal{N}\left(t,y\right).
\]

Comparing the forms \eqref{eq:formsfinal} and \eqref{eq:forms}, we see that, in this case, the coefficients $c_j$ are equal to zero and
\[
a_{kl}=\sum_{j=1}^{n}\left(\left(\frac{\partial h}{\partial y}\right)_{kj}^{-1}\left(t,y\right)\left(y\right)\right) \left(\left(\frac{\partial h}{\partial y}\right)_{lj}^{-1}\left(t,y\right)\left(1-c\left(t\right)\right)\mathcal{N}\left(t,y\right)\right).
\]

Using \eqref{eq:defnormal} and \eqref{eq:nunormal}. iv), we see that the conormal derivative is equal to
\begin{equation}
\begin{array}{rcl}\label{eq:cofinal}
&&\tilde{C}(t,y,D)v
\\
&=& \sum_{j=1}^{n}\sum_{k=1}^{n}\sum_{l=1}^{n}\left(\left(\frac{\partial h}{\partial y}\right)_{kj}^{-1}\left(t,y\right)\frac{\partial v}{\partial y_{k}}\left(y\right)\right)\left(\left(\frac{\partial h}{\partial y}\right)_{lj}^{-1}\left(t,y\right)\mathcal{N}\left(t,y\right)\left(1-c\left(t\right)\right)\nu_{l}\left(y\right)\right)
\\
&=&
\sum_{j=1}^{n}\sum_{k=1}^{n}\left(\frac{\partial h}{\partial y}\right)_{kj}^{-1}\left(t,y\right)\left(1-c\left(t\right)\right)n_{j}\left(t,h\left(t,y\right)\right)\frac{\partial v}{\partial y_{k}}\left(y\right),
 \quad y \in \partial \Omega.
\end{array}
\end{equation}

We conclude that the Dirichlet-to-Neumann operator associated with the forms (\ref{eq:formsfinal}) are the operators that take $g\in H^{\frac{1}{2}}\left(\partial\Omega\right)$ to the $H^{-\frac{1}{2}}\left(\partial\Omega\right)$-weak conormal derivative $\tilde{C}\left(t,x,D\right)u_{t}$, defined in (\ref{eq:cofinal}), where $u_{t}\in H^{1}\left(\Omega\right)$ is the unique solution of the Dirichlet problem:
\begin{equation}
\tilde{P}(t,y,D)u_{t}=\lambda u_{t}\left(y\right)+\sum_{j=1}^{n}\sum_{k=1}^{n}\sum_{l=1}^{n}\left(\frac{\partial h}{\partial y}\right)_{lj}^{-1}\left(t,y\right)\frac{\partial}{\partial y_{l}}\left(\left(\frac{\partial h}{\partial y}\right)_{kj}^{-1}\left(t,y\right)\frac{\partial u_{t}}{\partial y_{k}}\left(y\right)\right)  =  0, \,\,\,\, \gamma_{0}(u_{t})=g.
\end{equation}

The above discussion together with Theorem \ref{thm:maintheoremnonautonomouselliptic}
implies:
\begin{thm}
\label{thm:Resultadononcylindrical} Let $u_{0}\in L^{2}\left(\partial\Omega_{t_{0}}\right)$ and $f:\overline{S}\to\mathbb{C}$ be such that the map defined by $t\in\left[t_{0},\infty\right[\mapsto\left(y\in\partial\Omega\mapsto f\left(t,h\left(t,y\right)\right)\right)$ belongs to $C_{u}^{\alpha}\left(\left[t_{0},\infty\right[,L^{2}\left(\partial\Omega\right)\right)$, where, as always in this section, $\alpha\in\left]\frac{1}{2},1\right]$. Then, there exists a unique function $u:D\to\mathbb{C}$ such that the function $v$ defined by $v(t,y)=u(t,h(t,y))$ belongs to $C\left(\left[t_{0},\infty\right[,H^{1}\left(\Omega\right)\right)$ and satisfies:

1) $\gamma_{0}\left(v\right)\in C^{1}\left(\left]t_{0},\infty\right[,L^{2}\left(\partial\Omega\right)\right)\cap C\left(\left[t_{0},\infty\right[,L^{2}\left(\partial\Omega\right)\right)\cap C\left(\left]t_{0},\infty\right[,H^{\frac{1}{2}}\left(\partial\Omega\right)\right)$,

2) 
\begin{equation}
\begin{array}{rcl}
\tilde{P}\left(t,y,D\right)v\left(t\right)&=&0\\
\frac{d}{dt}\gamma_{0}(v)\left(t\right)&=&-\tilde{C}\left(t,y,D\right)v\left(t\right)+\tilde{f}\left(t\right),\,t\in\left]t_{0},\infty\right[\\
\gamma_{0}\left(v\right)\left(t_{0}\right)&=&v_{0}
\end{array}
\nonumber
\end{equation}
with $\tilde{f}\left(t,y\right)=f\left(t,h\left(t,y\right)\right)$ and $v_{0}\left(y\right)=u_{0}\left(h\left(t_{0},y\right)\right)$.

Moreover, if $\lim_{t\to\infty}\tilde{f}\left(t\right)=f_{\infty}$ in $L^{2}\left(\partial\Omega\right)$, then $\lim_{t\to\infty}v\left(t\right)=u_{\infty}$ in $H^{1}\left(\Omega\right)$, where $u_{\infty}$ is the unique solution of 
\begin{equation}
\left\{
\begin{gathered}
\left(\lambda+\Delta\right)u_{\infty}\left(x\right)  =  0, \quad x\in\Omega\\
\frac{\partial u_{\infty}}{\partial n}\left(x\right)  =  f_{\infty}\left(x\right),\quad x\in\partial\Omega
\end{gathered}
\right. .
\nonumber
\end{equation}
\end{thm}

It is clear, by our discussion, that $\tilde{P}\left(t,y,D\right)v\left(t\right)=0$ is equivalent to $\left(\lambda+\Delta\right)u\left(t\right)=0$, and \[\frac{d}{dt}\gamma_{0}(v)\left(t\right)=-\tilde{C}\left(t,y,D\right)v\left(t\right)+\tilde{f}\left(t\right)\] is formally equivalent to the dynamic boundary conditions of \eqref{eq:MainProblem2}, after change of variables.

\end{document}